\documentclass[10pt]{amsart}
\usepackage{amsfonts,amssymb,amscd,amsmath,enumerate,verbatim,calc}
\usepackage{mathrsfs}
\numberwithin{equation}{section}
\newtheorem{theorem}{Theorem}[section]
\newtheorem{corollary}[theorem]{Corollary}

\newtheorem{proposition}[theorem]{Proposition}
\newtheorem{definition}[theorem]{Definition}

\newtheorem{example}[theorem]{Example}
\newtheorem{remark}[theorem]{Remark}

\begin{document}

	\title[Exterior differential calculus in generalized Lie
	algebras (algebroids) category]
	{ Exterior differential calculus in generalized Lie
		algebras (algebroids) category with applications
		to interior and exterior algebraic (differential)
		systems }

	\bibliographystyle{amsplain}

	\author[Constantin M. Arcu\c{s}]{C. M. Arcu\c{s}}
	\address{ Secondary School "Cornelius
		Radu" Radine\c{s}ti Village, 217196, Gorj County, Romania.} \email{c\_arcus@radinesti.ro}

	\author[Esmaeil Peyghan]{E. Peyghan\\\\Dedicated to Professor Richard S. Palais}
	\address{Department of Mathematics, Faculty of Science, Arak University,
		Arak, 38156-8-8349, Iran.}
	\email{e-peyghan@araku.ac.ir}
	


	
\begin{abstract}
A new category of Lie algebras, called generalized Lie algebras, is
presented such that classical Lie algebras and Lie-Rinehart algebras are
objects of this new category. A new philosophy over generalized Lie
algebroids theory is presented using the notion of generalized Lie algebra
and examples of objects of the category of generalized Lie algebroids are
presented. An exterior differential calculus on generalized Lie algebras is
presented and a theorem of Maurer-Cartan type is obtained. Supposing that
any submodule (vector subbundle) of a generalized Lie algebra (algebroid) is an interior algebraic (differential) system (IAS (IDS)) for that generalized Lie algebra (algebroid), then the involutivity of the IAS (IDS) in a result of
Frobenius type is characterized. Introducing the notion of exterior
algebraic (differential) system of a generalized Lie algebra (algebroid), the
involutivity of an IAS (IDS) is characterized in a result of Cartan type.
Finally, new directions by research in algebraic (differential) symplectic
spaces theory are presented. 
\end{abstract}

\maketitle

\section{Introduction}

\ \ \ \ 

Throughout this paper a ring is a unitary ring and a module over a ring $\mathcal{F}$ is a left module (except for a commutative ring that we consider on a module as a left and right module). 

Lie groups and Lie algebras (as linearization of Lie groups) have a vast
importance in physics (for example in the classification of elementary
particles) \cite{I}. Lie algebras are related to Lie groups via two
approaches, first by geometrical bridge and second by fiber bundle theory.
Important applications of Lie algebras in physics and mechanics (see \cite%
{SW}) inspired many authors to study these spaces and generalized them to
other spaces such as Lie superalgebras (these spaces are important in
theoretical physics where they are used to describe the mathematics of
supersymmetry) \cite{K0}, affine (Kac-Moody) Lie algebras (these spaces play
an important role in string theory and conformal field theory) \cite{K},
quasisimple Lie algebras \cite{HT}, (locally) extended affine Lie algebras 
\cite{AABGP, MY, N} and invariant affine reflection algebras \cite{N1}.

The basic idea of some branches of mathematics and physics (in particular,
noncommutative geometry) is replacing the space $M$ by some algebra of
functions on it \cite{CO1}. Therefore, in mathematical physics, functions
and fields are more important than the manifolds on which they are defined 
\cite{LM}. Since the set of vector fields on $M$, i.e., $\chi (M)$, plays an
important role in differential geometry and $\chi (M)$ is the same as the
set of all derivations of smooth functions on $M$, i.e., $Der(\mathcal{F}%
(M)) $, then many authors such as Dubois-Violette generalized the algebra of
smooth functions to an arbitrary algebra $A$ and considered the Lie algebra $%
Der(A)$ of all derivations of $A$ as the generalization of the Lie algebra
of smooth vector fields \cite{D}. As the exterior differential operator,
interior product and Lie derivative are defined on $Der(\mathcal{F}(M))$ and
they have basic applications in mathematical physics, then it is important
to introduce these derivatives for $Der(A)$. This approach is well-known as
"generalization of differential calculus from classical differential
geometry to noncommutative geometry" \cite{CO, CO1, D}.

We know that an classical (usual) algebra over a commutative ring $\mathcal{F}$ is a $\mathcal{F}$-module $A$ for which there exists
a bilinear (biadditive and bihomogenous) operation 
\begin{equation*}
\begin{array}{ccc}
A\times A & ^{\underrightarrow{~\ \ [ ,] _{A}~\ \ }} & A \\ 
( u,v) & \longmapsto & [ u,v] _{A}%
\end{array}%
.
\end{equation*}

In this paper we are going to remove the condition of bihomogenity and we
consider an algebra over a ring (not only commutative) $\mathcal{F}$ as
being a $\mathcal{F}$-module $A$ for which there exists a
biaditive operation $[, ]_A:A\times A\rightarrow A$. Obviously an classical
$\mathcal{F}$-algebra is an $\mathcal{F}$-algebra in our direction, but the converse is not true. In the next, we present a Lie algebra over  $\mathcal{F}$ as being a $\mathcal{F}$-algebra $(A, [,]_A)$
such that the biadditive operation $[,]_A$ satisfies

$LA_{1}.$ $[ u,u] _{A}=0,$ for any $u\in A$,

$LA_{2}.$ $[u,[v,z]_{A}]_{A}+[z,[u,v]_{A}]_{A}+[v,[z,u]_{A}]_{A}=0,$ for any 
$u,v,z\in A.$\newline
Using the Lie $\mathcal{F}$-algebra $(Der(\mathcal{%
F}), [,]_{Der(\mathcal{F})})$ of derivations of $\mathcal{F}$ we introduce the notion of generalized Lie $\mathcal{F}$-algebra as being a $\mathcal{F}$-module $A$ such that there
exists a modules morphism $\rho $ from $A$ to $Der(\mathcal{F})$ and a biadditive operation $[,]_{A}:A\times A\rightarrow A$,
satisfying the following condition 
\begin{equation*}
\lbrack u,fv]_{A}=f[u,v]_{A}+\rho (u)(f)\cdot v,~\forall u,v\in A\text{ and }%
f\in \mathcal{F},
\end{equation*}
such that $(A, [,]_{A})$ is a Lie $\mathcal{F}$-algebra.

After presenting our definition of generalized Lie algebra we found that a
similar definition had been exhibited by Palais \cite{P} and Rinehart \cite%
{R}, which is called \textit{Lie d-ring} or \textit{Lie-Rinehart algebra} (see \cite{H, H1, H2} for more detailes). Recall that a pair $(A,\mathcal{F})$ is called a Lie-Rinehart algebra over $%
R $, where $R$ is a commutative and unitary ring, $(\mathcal{F}, [,]_{\mathcal{F}})$ is a commutative classical algebra over $R$, $(A, [,]_{A}) $ is a Lie algebra over $R$ and $A$ is the module over $%
\mathcal{F}$ such that there exists an $\mathcal{F}$-linear Lie algebras morphism $\rho $ from $%
(A, [,]_{A})$ to $(Der(\mathcal{F}), [,]_{Der(\mathcal{F})})$
satisfies in 
\begin{equation*}
\lbrack u,f\cdot v]_{A}=f\cdot \lbrack u,v]_{A}+\rho (u)(f)\cdot v,~\forall
u,v\in A\text{ and }f\in \mathcal{F}.
\end{equation*}%
Note that in the Lie-Rinehart algebra, the means of algebra (respectively,
Lie algebra) is the classical definition of algebra (respectively, Lie algebra).
It is easy to see that a Lie-Rinehart algebra $(A,\mathcal{F})$ is a
generalized Lie algebra over $\mathcal{F}$ (because $(A, [,]_{A})$
is a $\mathcal{F}$-algebra in our approach), but an arbitrary generalized
Lie algebra is not necessary a Lie-Rinehart algebra (see Example \ref{19}).

Using the notion of generalized Lie algebra, in Section 3 we present a new
philosophy over theory of generalized Lie algebroids. Some examples of objects of the
category of generalized Lie algebroids are presented. Using a similar method
used in \cite{A3, A4}, we propose in Section 4, a new point of view over
extension of exterior differential calculus from classical differential
geometry to noncommutative geometry using the notion of generalized Lie
algebra. In particular, using the locally generalized Lie algebras of an
arbitrary generalized Lie algebroid, we obtain an exterior differential
calculus for generalized Lie algebroids.

Using the \emph{Cartan's moving frame method}, there exists the following%
\newline
\textbf{Theorem} (E. Cartan). If $N\in \left\vert \mathbf{Man}%
_{n}\right\vert $ is a Riemannian manifold and $X_{\alpha }=X_{\alpha }^{i}%
\frac{\partial }{\partial x^{i}}$, $\alpha \in \overline{1,n}$ is an
orthonormal moving frame, then there exists a collection of $1$-forms $\Omega
_{\beta }^{\alpha },~\alpha ,\beta \in \overline{1,n}$ uniquely defined by
the requirements%
\begin{equation*}
\Omega _{\beta }^{\alpha }=-\Omega _{\alpha }^{\beta }
\end{equation*}%
\emph{and }%
\begin{equation*}
d^{F}\Theta ^{\alpha }=\Omega _{\beta }^{\alpha }\wedge \Theta ^{\beta
},~\alpha \in \overline{1,n}
\end{equation*}%
where $\left\{ \Theta ^{\alpha },\alpha \in \overline{1,n}\right\} $ is the
coframe (see \cite{Ni}, p. 151).

It is known that an $r$\emph{-dimensional distribution on a manifold }$N$ is
a mapping $\mathcal{D}$ defined on $N,$ which assigns to each point $x$ of $%
N $ an $r$-dimensional linear subspace $\mathcal{D}_{x}$ of $T_{x}N.$ A
vector fields $X$ belongs to $\mathcal{D}$ if we have $X_{x}\in \mathcal{D}%
_{x}$ for each $x\in N.$ When this happens we write $X\in \Gamma ( \mathcal{D%
})$. The distribution $\mathcal{D}$ on a manifold $N$ is said to be \emph{%
differentiable} if for any $x\in N$ there exist $r$ differentiable linearly
independent vector fields $X_{1},\cdots ,X_{r}\in \Gamma ( \mathcal{D}) $ in
a neighborhood of $x.$ The distribution $\mathcal{D}$ is said to be \emph{%
involutive }if for all vector fields $X,Y\in \Gamma ( \mathcal{D}) $ we have 
$[ X,Y] \in \Gamma ( \mathcal{D}) .$

In the classical theory we have the following\newline
\textbf{Theorem} (Frobenius) The distribution $\mathcal{D}$ is involutive if
and only if for each $x\in N$ there exists a neighborhood $U$ and $n-r$
linearly independent $1$-forms $\Theta ^{r+1},\cdots ,\Theta ^{n}$ on $U$
which vanish on $\mathcal{D}$ and satisfy the condition 
\begin{equation*}
d^{F}\Theta ^{\alpha }=\Sigma _{\beta \in \overline{r+1,p}}\Omega _{\beta
}^{\alpha }\wedge \Theta ^{\beta },~\alpha \in \overline{r+1,n},
\end{equation*}%
for suitable $1$-forms $\Omega _{\beta }^{\alpha },~\alpha ,\beta \in 
\overline{r+1,n}$ (see \cite{L}, p. $58$).

In Section 5, we introduce the definition of an interior
algebraic (differential) system (IAS (IDS)) of a generalized Lie
algebra (algebroid) and a characterization of the ivolutivity of an IAS (IDS) in
a result of Frobenius type is presented in Theorem \ref{Frob} and Corollary %
\ref{CFrob}, respectively. In the classical sense, an exterior differential
system (EDS) is a pair $(M,\mathcal{I})$ consisting of a smooth manifold $M$
and a homogeneous, differentially closed ideal $\mathcal{I}$ in the algebra
of smooth differential forms on $M$ (see \cite{BCGGG, G, IL, Ka}). Extending
the classical notion of EDS to notion of exterior algebraic (differential)
system (EAS (EDS)) of a generalized Lie algebra (algebroid), then the
involutivity of an IAS (IDS) in a result of Cartan type is presented in the
Theorem \ref{Cartan} and Corollary \ref{CCartan} respectively. Indeed, in
this section we show that there exists very close links between EAS (EDS) and the noncommutative geometry of generalized Lie algebras (algebroids).
Finally, in Section 6, we present new directions by research in symplectic
noncommutative geometry. 

\section{Generalized Lie algebras}

\ \ \ \ 
In this section, we introduce the generalized Lie $\mathcal{F}$-algebras
category and we present some examples of objects of this category. Also, we
obtain some properties of objects of this new category.
\begin{definition}
If $A$ is a $\mathcal{F}$-module such that there exists an
biaditive operation 
\begin{equation*}
\begin{array}{ccc}
A\times A & ^{\underrightarrow{~\ \ [ ,] _{A}~\ \ }} & A \\ 
( u,v) & \longmapsto & [ u,v] _{A}%
\end{array}
,
\end{equation*}%
then we say that $( A, [ ,] _{A}) $ is a $\mathcal{F}$-algebra or algebra over $\mathcal{F}$.
\end{definition}
If $(A, [,]_A)$ is a $\mathcal{F}$-algebra such that the operation $%
[,]_A$ is associative (commutative), then  $(A, [,]_A)$ is called an
associative (commutative) $\mathcal{F}$-algebra. Moreover, $(A, [,]_A)$ is called a unitary $\mathcal{F}$-algebra, if the
operation $[,]_A$ has a unitary element.
\begin{remark}
In the above definition, if $\mathcal{F}$ is a commutative ring and $[,]_A$ is bilinear, then we have the classical definition
of algebra over a ring. Thus every classical $\mathcal{F}$-algebra is an $\mathcal{F}$-algebra, but the converse is not true. For example if $M$
is a manifold, then $(\chi(M), [,])$, where 
\begin{equation*}
[X, Y](f)=X(Y(f))-Y(X(f)),\ \ \ \forall X, Y\in\chi(M),\ \ \forall f\in \mathcal{F}(M),
\end{equation*}
 is an $\mathcal{F}(M)$-algebra
but it is not a classical $\mathcal{F}(M)$-algebra (only a classical $%
\mathbb{R}$-algebra, where $\mathbb{R}$ is the field of real
numbers).
\end{remark}
\begin{definition}
If $\mathcal{F}$ is a ring, then the set $Der( 
\mathcal{F}) $\ of groups morphisms $X:\mathcal{F\longrightarrow F}$
satisfying the condition 
\begin{equation*}
X( f\cdot g) =X( f) \cdot g+f\cdot X( g) ,~\forall f,g\in \mathcal{F},
\end{equation*}
will be called the set of derivations of  $\mathcal{F}$.
\end{definition}
If $M$ is a manifold, then it is easy to check that $Der( \mathcal{F}( M) )=\chi(M)$.
\begin{example}
If we consider the
biadditive operation 
\begin{equation*}
\begin{array}{ccc}
Der( \mathcal{F}) \times Der( \mathcal{F}) & ^{\underrightarrow{~\ \ [
,]_{Der(\mathcal{F})} ~\ \ }} & Der( \mathcal{F}) \\ 
( X,Y) & \longmapsto & [ X,Y]_{Der(\mathcal{F})}%
\end{array}
,
\end{equation*}%
given by 
\begin{equation*}
[ X,Y]_{Der(\mathcal{F})}( f) =X( Y( f) ) -Y( X( f) ) ,~\forall f\in 
\mathcal{F},
\end{equation*}
then $( Der( \mathcal{F})  ,[ ,]_{Der(\mathcal{F})} ) $ is an $\mathcal{F}$-algebra.
\end{example}
\begin{definition}\label{Def12}
If $( A,[ ,] _{A}) $ is an $\mathcal{F}$-algebra such that $[ ,]
_{A} $ satisfies the conditions:

$LA_{1}.$ $[ u,u] _{A}=0,$ for any $u\in A$,

$LA_{2}.$ $[ u,[ v,z] _{A}] _{A}+[ z,[ u,v] _{A}] _{A}+[ v,[ z,u] _{A}]
_{A}=0,$ for any $u,v,z\in A$,\newline
then we will say that $( A,[ ,] _{A}) $ is a Lie $\mathcal{F}$%
-algebra or Lie algebra over $\mathcal{F}$.
\end{definition}
\begin{example}
It is easy to check that the $\mathcal{F}$-algebra $( Der( \mathcal{F})
,[,]_{Der(\mathcal{F})}) $ is a Lie $\mathcal{F}$-algebra.
\end{example}
Using Definition \ref{Def12} we deduce the following:
\begin{proposition}
If $( A,[ ,] _{A}) $ is a Lie $\mathcal{F}$-algebra, then we have:

1. $[ u,v] _{A}=-[ v,u] _{A},$ for any $u,v\in A$,

2. $[ u,0] _{A}=0,$ for any $u\in A$,

3. $[ -u,v] _{A}=-[ u,v] _{A}=[ u,-v] _{A}$, for any $u,v\in A.$
\end{proposition}

\begin{definition}
Let $ A $ be an $\mathcal{F}$-module. If there exists a modules
morphism $\rho $ from $ A$ to $Der( \mathcal{F}) $
and a biadditive operation 
\begin{equation*}
\begin{array}{ccc}
A\times A & ^{\underrightarrow{~\ \ [ ,] _{A}~\ \ }} & A \\ 
( u,v) & \longmapsto & [ u,v] _{A}%
\end{array}
,
\end{equation*}%
satisfies in 
\begin{equation}
[ u,fv] _{A}=f[ u,v] _{A}+\rho ( u) ( f) \cdot v,
\end{equation}%
for any $u,v\in A$ and $f\in \mathcal{F}$ such that $( A,[ ,]
_{A}~) $ is a Lie $\mathcal{F}$-algebra, then $( A ,[
,] _{A}, \rho ) $ is called a generalized Lie $\mathcal{F}$-algebra or
generalized Lie algebra over $ \mathcal{F}$.
\end{definition}
\begin{example}\label{19}
Let $\mathcal{F}$ be a commutative ring. Obviously it is an $\mathcal{F}$-module. Now we consider the direct sum 
\begin{equation*}
Der(\mathcal{F})\oplus\mathcal{F}=\{X\oplus f| X\in Der(\mathcal{F}), f\in%
\mathcal{F}\}.
\end{equation*}
If we define 
\begin{equation*}
(X\oplus f)+(Y\oplus g)=(X+Y)\oplus(f+g),\ \ h\cdot(X\oplus f)=h\cdot
X\oplus h\cdot f,
\end{equation*}
for any $X, Y\in Der(\mathcal{F})$ and $f, g, h\in\mathcal{F}$, then $Der(%
\mathcal{F})\oplus\mathcal{F}$ is a $\mathcal{F}$-module.
Defining 
\begin{align} \label{LIE}
[X\oplus f, Y\oplus g]_{_{Der(\mathcal{F})\oplus\mathcal{F}}}=[X, Y]_{_{Der(%
\mathcal{F})}}\oplus (X(g)-Y(f)),
\end{align}
it is easy to see that $[, ]_{_{Der(\mathcal{F})\oplus\mathcal{F}}}$ is
biadditive on $Der(\mathcal{F})\oplus\mathcal{F}$ and so $(Der(\mathcal{F}%
)\oplus\mathcal{F},  [, ]_{_{Der(\mathcal{F})\oplus\mathcal{F}}})$
is an $\mathcal{F}$-algebra. Direct calculations give us 
\begin{align*}
[X\oplus f, X\oplus f]_{_{Der(\mathcal{F})\oplus\mathcal{F}}}=0_{Der(%
\mathcal{F})}\oplus 0_{\mathcal{F}}=0_{Der(\mathcal{F})\oplus\mathcal{F}},
\end{align*}
and 
\begin{align*}
&[X\oplus f, [Y\oplus g, Z\oplus h]]_{_{Der(\mathcal{F})\oplus\mathcal{F}%
}}+[Y\oplus g, [Z\oplus h, X\oplus f]]_{_{Der(\mathcal{F})\oplus\mathcal{F}}}
\\
&+[Z\oplus h, [X\oplus f, Y\oplus g]]_{_{Der(\mathcal{F})\oplus\mathcal{F}%
}}=0_{Der(\mathcal{F})\oplus\mathcal{F}}.
\end{align*}
Thus $(Der(\mathcal{F})\oplus\mathcal{F}, [, ]_{_{Der(\mathcal{F}
)\oplus\mathcal{F}}})$ is a Lie $\mathcal{F}$-algebra. Now we define 
\begin{equation*}
\begin{array}{ccc}
Der(\mathcal{F})\oplus\mathcal{F} & ^{\underrightarrow{~\ \ \rho~\ \ }} & 
Der(\mathcal{F}) \\ 
X\oplus f & \longmapsto & X%
\end{array}
,
\end{equation*}%
for any $X\in Der(\mathcal{F})$ and $f\in\mathcal{F}$. It is easy to check
that $\rho$ is a modules morphism from $Der(\mathcal{F})\oplus\mathcal{F}
$ to $Der(\mathcal{F})$. Here, we show that 
\begin{align}\label{LIE1}
&[X\oplus f, h\cdot(Y\oplus g)]_{_{Der(\mathcal{F})\oplus\mathcal{F}%
}}=h\cdot [X\oplus f, Y\oplus g]_{_{Der(\mathcal{F})\oplus\mathcal{F}%
}}+\rho(X\oplus f)(h)\cdot(Y\oplus g).
\end{align}
Using (\ref{LIE}) we get 
\begin{align*}
&[X\oplus f, h\cdot(Y\oplus g)]_{_{Der(\mathcal{F})\oplus\mathcal{F}%
}}=[X\oplus f, h\cdot Y\oplus h\cdot g]_{_{Der(\mathcal{F})\oplus\mathcal{F}%
}}=[X, h\cdot Y]_{_{Der(\mathcal{F})}}  \notag \\
&\oplus(X(h\cdot g)-h\cdot Y(f))=(h\cdot [X, Y]_{_{Der(\mathcal{F}%
)}}+X(h)\cdot Y)\oplus(X(h\cdot g)-h\cdot Y(f))  \notag \\
&=h\cdot [X, Y]_{_{Der(\mathcal{F})}}\oplus h\cdot(X(g)-Y(f))+X(h)\cdot
Y\oplus g\cdot X(h) \\
&=h\cdot ([X, Y]_{_{Der(\mathcal{F})}}\oplus(X(g)-Y(f)))+X(h)\cdot(Y\oplus g)
\\
&=h\cdot [X\oplus f, Y\oplus g]_{_{Der(\mathcal{F})\oplus\mathcal{F}%
}}+\rho(X\oplus f)(h)\cdot(Y\oplus g).
\end{align*}
Thus (\ref{LIE1}) holds and consequently $(Der(\mathcal{F})\oplus\mathcal{F},
 [, ]_{_{Der(\mathcal{F})\oplus\mathcal{F}}}, \rho)$ is a
generalized Lie $\mathcal{F}$-algebra. Relation (\ref{LIE}) shows that $[, ]_{_{Der(\mathcal{F})\oplus\mathcal{F}}}$ is not $\mathcal{F}$-bilinear and consequently $(Der(\mathcal{F})\oplus\mathcal{F},
 [, ]_{_{Der(\mathcal{F})\oplus\mathcal{F}}})$ is not classical Lie algebra over $\mathcal{F}$. Thus $(Der(\mathcal{F%
})\oplus\mathcal{F}, \mathcal{F})$ is not a Lie-Rinehart algebra over $%
\mathcal{F}$.
\end{example}
\begin{proposition}
\label{Prop18} If $( A,[ ,] _{A}, \rho ) $ is a generalized Lie $%
\mathcal{F}$-algebra, then $\rho $ is a Lie algebras morphism from $%
(A ,[ ,] _{A})$ to $( Der( \mathcal{F}) ,[ ,]_{Der( 
\mathcal{F})})$.
\end{proposition}
\begin{proof}
Let $u,v,w\in A$ and $f\in \mathcal{F}$. Since $(A
,[,]_{A})$ is a Lie $\mathcal{F}$-algebra, we have the Jacobi identity:%
\begin{equation}
\lbrack \lbrack u,v]_{A},fw]_{A}=[u,[v,fw]_{A}]_{A}-[v,[u,fw]_{A}]_{A}.
\label{Jacob}
\end{equation}%
Using the definition of generalized Lie $\mathcal{F}$-algebra, we obtain 
\begin{equation*}
\lbrack \lbrack u,v]_{A},fw]_{A}=f[[u,v]_{A},w]_{A}+\rho \lbrack
u,v]_{A}(f)\cdot w.
\end{equation*}%
Similarly, we get 
\begin{equation*}
\begin{array}{cl}
\lbrack u,[v,fw]_{A}]_{A} & =[u,f[v,w]_{A}+\rho (v)(f)\cdot w]_{A} \\ 
& =f[u,[v,w]_{A}]_{A}+\rho (u)(f)\cdot \lbrack v,w]_{A} \\ 
& \ \ \ +\rho (v)(f)\cdot \lbrack u,w]_{A}+\rho (u)(\rho (v)(f))\cdot w,%
\end{array}%
\end{equation*}%
and%
\begin{equation*}
\begin{array}{cl}
\lbrack v,[u,fw]_{A}]_{A} & =[v,f[u,w]_{A}+\rho (u)(f)\cdot w]_{A} \\ 
& =f[v,[u,w]_{A}]_{A}+\rho (v)(f)\cdot \lbrack u,w]_{A} \\ 
& \ \ \ +\rho (u)(f)\cdot \lbrack v,w]_{A}+\rho (v)(\rho (u)(f))\cdot w.%
\end{array}%
\end{equation*}%
Setting three above equations in (\ref{Jacob}) and using Jacobi identity, we
obtain 
\begin{equation*}
\begin{array}{cl}
\rho \lbrack u,v]_{A}(f)\cdot w & =\rho (u)(\rho (v)(f))\cdot w-\rho
(v)(\rho (u)(f))\cdot w \\ 
& =[\rho (u),\rho (v)]_{Der( \mathcal{F})}(f)\cdot w,%
\end{array}%
\end{equation*}%
and consequently 
\begin{equation*}
\rho \lbrack u,v]_{A}=[\rho (u),\rho (v)]_{Der( \mathcal{F})}.
\end{equation*}%
Thus modules morphism $\rho $ is a Lie algebras morphism from $(A
,[,]_{A})$ to\\ $(Der(\mathcal{F}) ,[,]_{Der(\mathcal{F})})$.
\end{proof}
\begin{definition}
Let $( A^{\prime } ,[ ,] _{A^{\prime }},\rho ^{\prime }) $ be an
another generalized Lie $\mathcal{F}$-algebra. A generalized Lie $\mathcal{F}
$-algebras morphism from $( A,[ ,] _{A}, \rho ) $ to $( A^{\prime
} ,[ ,] _{A^{\prime }},\rho ^{\prime }) $ is a Lie $\mathcal{F}$%
-algebras morphism $\varphi $ from $( A ,[ ,] _{A})$ to $( A^{\prime
},[ ,] _{A^{\prime }}) $ such that the following diagram is
commutative:%
\begin{equation*}
\begin{array}{ccc}
A & ^{\underrightarrow{~\ \ \varphi ~\ \ }} & ~\ \ \ \ \ A^{\prime } \\ 
~\ \ \ \ \rho \searrow &  & \swarrow \rho ^{\prime } \\ 
& Der( \mathcal{F}) & 
\end{array}%
.
\end{equation*}
\end{definition}
Thus, we can discuss about the category of generalized Lie $\mathcal{F}$%
-algebra.
\begin{remark}
When we consider a basis for $A$, we let that it is free module and when we consider a basis for $Der(\mathcal{F}) $, we let
that it is a free module also. Note that, in general, $Der(\mathcal{F})$ is
not a free module. For example, if $M$ is an arbitrary manifold then $Der(%
\mathcal{F}(M))=\chi (M)$ is not free module. But, if $M$ is a
parallelizable manifold, then $Der(\mathcal{F}(M))=\chi (M)$ is a free
module. In particular, if $G$ is a Lie group then $Der(\mathcal{F}(G))=\chi
(G)$ is a free module.
\end{remark}

Let $\{\partial _{i}\}$ be a basis of the Lie $\mathcal{F}$-algebra of
derivations of $\mathcal{F}$ and $\{t_{\alpha }\}$
be a basis of the generalized Lie $\mathcal{F}$-algebra $(A
,[,]_{A},\rho )$. Then we use the notations%
\begin{equation*}
\rho (t_{\alpha })=\rho _{\alpha }^{i}\partial _{i},\ \ \ 
\lbrack t_{\alpha },t_{\beta }]_{A}=L_{\alpha \beta }^{\gamma }t_{\gamma },\ \ \ [\partial_i, \partial_j]_{Der(\mathcal{F})}=C_{ij}^k\partial_k,
\end{equation*}%
where $\rho _{\alpha }^{i}$, $L_{\alpha \beta }^{\gamma }$ and $C_{ij}^k$ belong to $\mathcal{F}$. It is easy to see that $L_{\alpha \beta }^{\gamma }=-L_{\beta \alpha
}^{\gamma }$ and $C_{ij}^k=-C_{ji}^k$. The components $L_{\alpha \beta }^{\gamma }$ are called the {%
\textit{structure elements}} of the generalized Lie $\mathcal{F}$-algebra $%
(A ,[,]_{A}~,\rho )$. Using the above proposition, we obtain 
\begin{equation*}
\begin{array}{c}
L_{\alpha \beta }^{\gamma }\rho _{\gamma }^{k}=\rho _{\alpha }^{i}\partial
_{i}(\rho _{\beta }^{k})-\rho _{\beta }^{j}\partial _{j}(\rho _{\alpha
}^{k})+\rho^i_\alpha\rho^j_\beta C_{ij}^k.
\end{array}%
\end{equation*}
Here considering a basis for $Der(\mathcal{F})$ we present a new example of generalized Lie algebra.
\begin{example}
	Let $\mathcal{F}$ be a ring such that $Der(\mathcal{F})$ is a free module with basis $\{\partial_i\}_{i\in
		\overline{1,m}}$. Using the modules morphism $\rho:Der(\mathcal{F})\rightarrow Der(\mathcal{F})$, we define the operation
	\begin{equation*}
	\begin{array}{ccc}
	Der(\mathcal{F})\times Der(\mathcal{F}) & ^{\underrightarrow{~\ \ \bullet ~\ \ }} & Der(\mathcal{F}) \\
	(X,Y) & \longmapsto  & X\bullet Y
	\end{array}%
	,
	\end{equation*}
	given by the equality $X\bullet Y=Y^i(X\circ \partial_i)+\rho(X)(Y^i)\partial_i$, where $Y=Y^i\partial_i$ and "$\circ$" is the usual composition operation. Now, we define 
	\begin{equation}\label{*1}
	[X,Y]^{\bullet}_{Der(\mathcal{F})}=X\bullet Y-Y\bullet X, \ \ \forall X, Y\in Der(\mathcal{F}).
	\end{equation}
	If $X=X^i\partial_i$ and $Y=Y^j\partial_j$, then we have
	\begin{align*}
	&[X,Y]^{\bullet}_{Der(\mathcal{F})}(fg)=[X^i\partial_i,Y^j\partial_j]^{\bullet}_{Der(\mathcal{F})}(fg)=\{(X^i\partial_i)\bullet(Y^j\partial_j)-(Y^j\partial_j)\bullet (X^i\partial_i)\}(fg)\\
	&=\{X^iY^j\partial_i\circ\partial_j+X^i\rho(\partial_i)(Y^j)\partial_j-Y^j\rho(\partial_j)(X^i)\partial_i-X^iY^j\partial_j\circ\partial_i\}(fg). 
	\end{align*}
	Using the Liebnitz property of $\partial_i$ and $\partial_j$ we get
	\begin{align*}
	&[X,Y]^{\bullet}_{Der(\mathcal{F})}(fg)=\{X^iY^j(\partial_i\partial_j(f))+X^i\rho(\partial_i)(Y^j)\partial_j(f)-Y^j\rho(\partial_j)(X^i)\partial_i(f)\\
	&-X^iY^j(\partial_j\partial_i(f))\}\cdot g+f\cdot\{X^iY^j(\partial_i\partial_j(g))+X^i\rho(\partial_i)(Y^j)\partial_j(g)-Y^j\rho(\partial_j)(X^i)\partial_i(g)\\
	&-X^iY^j(\partial_j\partial_i(g))\}=[X^i\partial_i,Y^j\partial_j]^{\bullet}_{Der(\mathcal{F})}(f)\cdot g+f\cdot[X^i\partial_i,Y^j\partial_j]^{\bullet}_{Der(\mathcal{F})}(g)\\
	&=[X,Y]^{\bullet}_{Der(\mathcal{F})}(f)\cdot g+f\cdot[X,Y]^{\bullet}_{Der(\mathcal{F})}(g).
	\end{align*}
	Thus $[X,Y]^{\bullet}_{Der(\mathcal{F})}\in Der(\mathcal{F})$. Moreover, if $f\in\mathcal{F}$, then we have
	\begin{align*}
	&[X,fY]^{\bullet}_{Der(\mathcal{F})}=[X^i\partial_i,fY^j\partial_j]^{\bullet}_{Der(\mathcal{F})}=(X^i\partial_i)\bullet(fY^j\partial_j)-(fY^j\partial_j)\bullet(X^i\partial_i)\\
	&=fX^iY^j\partial_i\circ\partial_j+X^i\rho(\partial_i)(f)Y^j\partial_j+fX^i\rho(\partial_i)(Y^j)\partial_j-fX^iY^j(\partial_j\circ\partial_i)\\
	&-fY^j\rho(\partial_j)(X^i)\partial_i=f[X, Y]^{\bullet}_{Der(\mathcal{F})}+\rho(X)(f)Y.
	\end{align*}
	Now we check the Jacobi identity for $[,]_{Der(\mathcal{F})}$. Using (\ref{*1}) we get
	\begin{align*}
	&[[X,Y]^{\bullet}_{Der(\mathcal{F})}, Z]^{\bullet}_{Der(\mathcal{F})}=[[X^i\partial_i,Y^j\partial_j]^{\bullet}_{Der(\mathcal{F})}, Z^k\partial_k]^{\bullet}_{Der(\mathcal{F})}\\
	&=[(X^i\partial_i)\bullet(Y^j\partial_j)-(Y^j\partial_j)\bullet(X^i\partial_i),Z^k\partial_k]^{\bullet}_{Der(\mathcal{F})}\\
	&=[X^iY^j\partial_i\circ\partial_j+X^i\rho(\partial_i)(Y^j)\partial_j-X^iY^j(\partial_j\circ\partial_i)-Y^j\rho(\partial_j)(X^i)\partial_i,Z^k\partial_k]^{\bullet}_{Der(\mathcal{F})}\\
	&=X^iY^jZ^k\partial_i\circ\partial_j\circ\partial_k+X^iY^j\rho(\partial_i)(\rho(\partial_j)(Z^k))\partial_k-Z^k\rho(\partial_k)(X^iY^j)\partial_i\circ\partial_j-X^iY^jZ^k\partial_k\circ\partial_i\circ\partial_j\\
	&+X^iZ^k\rho(\partial_i)(Y^j)\partial_j\circ\partial_k+X^i\rho(\partial_i)(Y^j)\rho(\partial_j)(Z^k)\partial_k-X^iZ^k\rho(\partial_i)(Y^j)\partial_k\circ\partial_j-Z^k\rho(\partial_k)(X^i\rho(\partial_i)(Y^j))\partial_j\\
	&-X^iY^jZ^k\partial_j\circ\partial_i\circ\partial_k-X^iY^j\rho(\partial_j)(\rho(\partial_i)(Z^k))\partial_k+Z^k\rho(\partial_k)(X^iY^j)\partial_j\circ\partial_i+X^iY^jZ^k\partial_k\circ\partial_j\circ\partial_i\\
	&-Y^jZ^k\rho(\partial_j)(X^i)\partial_i\circ\partial_k-Y^j\rho(\partial_j)(X^i)\rho(\partial_i)(Z^k)\partial_k+Y^jZ^k\rho(\partial_j)(X^i)\partial_k\circ\partial_i+Z^k\rho(\partial_k)(Y^j\rho(\partial_j)(X^i))\partial_i.
	\end{align*}
	Also we can obtain similar relations for $[[Y,Z]^{\bullet}_{Der(\mathcal{F})}, X]^{\bullet}_{Der(\mathcal{F})}$ and $[[Z,X]^{\bullet}_{Der(\mathcal{F})}, Y]^{\bullet}_{Der(\mathcal{F})}$. Using these relations it is easy to see that  
	\[
	[[X,Y]^{\bullet}_{Der(\mathcal{F})}, Z]^{\bullet}_{Der(\mathcal{F})}+[[Y,Z]^{\bullet}_{Der(\mathcal{F})}, X]^{\bullet}_{Der(\mathcal{F})}+[[Z,X]^{\bullet}_{Der(\mathcal{F})}, Y]^{\bullet}_{Der(\mathcal{F})}=0.
	\] 
	Thus Jacobi identity holds for this bracket. So, $(Der(\mathcal{F}), [,]^{\bullet}_{Der(\mathcal{F})}, \rho)$ is a new object of the category of generalized Lie algebras
\end{example}

\section{Generalized Lie algebroids}

\ \ \ \ 
Let $(E,\pi ,M)$ be an arbitrary vector bundle such that $M$ is paracompact
and let $[ \mathcal{A}_{M}] $ be the differentiable structure of te base
manifold $M$, such that for every $( U,\mathcal{\xi }_{U}) \in [ \mathcal{A}%
_{M}] $, the set $U$ is also the domain of a locally bundle chart.

For any $( U,\mathcal{\xi }_{U}) \in [ \mathcal{A}_{M}]$, it is easily seen that $\mathcal{F}( M) _{\mid U}=\left\{ f_{\mid U}~;~f\in \mathcal{F}( M) \right\} 
$ is a unitary ring with respect to the operations
\begin{equation*}
\begin{array}{ccc}
\mathcal{F}( M) _{\mid U}\mathcal{\times F}( M) _{\mid U} & ^{%
\underrightarrow{~\ \ +~\ \ }} & \mathcal{F}( M) _{\mid U} \\ 
( f_{\mid U},g_{\mid U}) & \longmapsto & f_{\mid U}+g_{\mid U}%
\end{array}
,
\end{equation*}%
and%
\begin{equation*}
\begin{array}{ccc}
\mathcal{F}( M) _{\mid U}\mathcal{\times F}( M) _{\mid U} & ^{%
\underrightarrow{~\ \ \cdot ~\ \ }} & \mathcal{F}( M) _{\mid U} \\ 
( f_{\mid U},g_{\mid U}) & \longmapsto & f_{\mid U}\cdot g_{\mid U}%
\end{array}
.
\end{equation*}
We remark that any function $f\in \mathcal{F}( M) $ is given by the set of
all its restrictions $\left\{ f_{\mid U},~( U,\mathcal{\xi }_{U}) \in [ 
\mathcal{A}_{M}] \right\}$. It follows easily that $\mathcal{F}( M) $ is a unitary
ring with respect to the operations%
\begin{equation*}
\begin{array}{ccc}
\mathcal{F}( M) \mathcal{\times F}( M) & ^{\underrightarrow{~\ \ +~\ \ }} & 
\mathcal{F}( M) \\ 
( f,g) & \longmapsto & f+g%
\end{array}
,
\end{equation*}%
and%
\begin{equation*}
\begin{array}{ccc}
\mathcal{F}( M) \mathcal{\times F}( M) & ^{\underrightarrow{~\ \ \cdot ~\ \ }%
} & \mathcal{F}( M) \\ 
( f,g) & \longmapsto & f\cdot g%
\end{array}
,
\end{equation*}%
such that $( f+g) _{\mid U}=f_{\mid U}+g_{\mid U}$ and $( f\cdot g) _{\mid
U}=f_{\mid U}\cdot g_{\mid U},$ for any $( U,\mathcal{\xi }_{U}) \in [ 
\mathcal{A}_{M}]$.\\
We now consider the
restriction vector bundle $( E_{\mid U},\pi ,U)$, for any $( U,\mathcal{\xi }_{U}) \in [ \mathcal{A}_{M}]$. We know that $\Gamma (
E_{\mid U},\pi ,U) $ is a free module over
$\mathcal{F}( M) _{\mid U}$. For any $( U,\mathcal{\xi }_{U})
\in [ \mathcal{A}_{M}] ,$ we derive that $\Gamma ( E_{\mid U},\pi ,U) $ is a 
$\mathcal{F}( M) _{\mid U}$-module with respect to the operations%
\begin{equation*}
\begin{array}{ccc}
\Gamma ( E_{\mid U},\pi ,U) \mathcal{\times }\Gamma ( E_{\mid U},\pi ,U) & ^{%
\underrightarrow{~\ \ +~\ \ }} & \Gamma ( E_{\mid U},\pi ,U) \\ 
( u_{\mid U},v_{\mid U}) & \longmapsto & u_{\mid U}+v_{\mid U}%
\end{array}%
,
\end{equation*}%
\begin{equation*}
\begin{array}{ccc}
\mathcal{F}( M) _{\mid U}\mathcal{\times }\Gamma ( E_{\mid U},\pi ,U) & ^{%
\underrightarrow{~\ \ \cdot ~\ \ }} & \Gamma ( E_{\mid U},\pi ,U) \\ 
( f_{\mid U},u_{\mid U}) & \longmapsto & f_{\mid U}\cdot u_{\mid U}%
\end{array}
.
\end{equation*}
Any global section $u\in \Gamma (E,\pi ,M)$ is given by the set of all its
restrictions $\left\{ u_{\mid U},~( U,\varphi _{U}) \in [ \mathcal{A}_{M}]
\right\}$. It is easy to check that $\Gamma (E,\pi ,M)$ is a $\mathcal{F}( M) $-module
with respect to the operations%
\begin{equation*}
\begin{array}{ccc}
\Gamma (E,\pi ,M)\mathcal{\times }\Gamma (E,\pi ,M) & ^{\underrightarrow{~\
\ +~\ \ }} & \Gamma (E,\pi ,M) \\ 
( u,v) & \longmapsto & u+v%
\end{array}
,
\end{equation*}%
and%
\begin{equation*}
\begin{array}{ccc}
\mathcal{F}( M) \mathcal{\times }\Gamma (E,\pi ,M) & ^{\underrightarrow{~\ \
\cdot ~\ \ }} & \Gamma (E,\pi ,M) \\ 
( f,u) & \longmapsto & f\cdot u%
\end{array}
,
\end{equation*}%
such that $( u+v) _{\mid U}=u_{\mid U}+v_{\mid U}$ and $( f\cdot u) _{\mid
U}=f_{\mid U}\cdot u_{\mid U},$ for any $( U,\xi _{U}) \in [ \mathcal{A}%
_{M}] $.\\
Now, let $(F,\nu ,N)$ be an another vector bundle and $(\varphi ,\varphi
_{0})$ is a vector bundles morphism from $(E,\pi ,M)$ to $(F,\nu ,N)$ such
that $\varphi _{0}$ is a diffeomorphism from $M$ to $N$. Using the operation 
\begin{equation*}
\begin{array}{ccc}
\mathcal{F}(M)\times \Gamma (F,\nu ,N) & ^{\underrightarrow{~\ \ \cdot ~\ \ }%
} & \Gamma (F,\nu ,N), \\ 
(f,z) & \longmapsto & f\circ \varphi _{0}^{-1}\cdot z%
\end{array}
,
\end{equation*}%
we deduce that $\Gamma (F,\nu ,N)$ is a $\mathcal{F}(M)$-module.

Let $( U,\xi _{U}) \in [ \mathcal{A}_{M}] $ and $( V,\eta _{V}) \in [ 
\mathcal{A}_{N}] $ such that $U\subseteq \varphi _{0}^{-1}(V)$ and let $\left\{ s_{a},a\in \overline{1,r}\right\}$ and $\left\{ t_{\alpha
},\alpha \in \overline{1,p}\right\} $ be the local bases for $\Gamma ( E_{\mid U},\pi ,U) $ and $ \Gamma (
F_{V},\nu ,V)$, respectively. Without restriction of generality, we can
consider that locally we have the modules morphism 
\begin{equation*}
\begin{array}{ccc}
\Gamma ( E_{\mid U},\pi ,U) & ^{\underrightarrow{~\ \ \Gamma ( \varphi
,\varphi _{0}) ~\ \ }} & \Gamma ( F_{\mid V},\nu ,V) \\ 
u_{\mid U}=u^{a}s_{a} & \longmapsto & \Gamma ( \varphi ,\varphi _{0}) (
u_{\mid U})%
\end{array}
,
\end{equation*}%
where 
$\Gamma ( \varphi ,\varphi _{0}) ( u_{\mid U}) =( u^{a}\circ \varphi
_{0}^{-1}) ( \varphi _{a}^{\alpha }\circ \varphi _{0}^{-1}) t_{\alpha }$.
\begin{proposition}
Let $(\rho ,\eta )$ and $( Th,h) $ be two vector bundles morphisms given by
the diagram
\begin{equation*}
\begin{array}[b]{ccccc}
F & ^{\underrightarrow{~\ \ \ \rho \ \ \ \ }} & TM & ^{\underrightarrow{~\ \
\ Th\ \ \ \ }} & TN \\ 
~\downarrow \nu &  & ~\ \ \downarrow \tau _{M} &  & ~\ \ \downarrow \tau _{N}
\\ 
N & ^{\underrightarrow{~\ \ \ \eta ~\ \ }} & M & ^{\underrightarrow{~\ \ \
h~\ \ }} & N%
\end{array}
,
\end{equation*}%
such that $\eta $ and $h$ are diffeomorphisms. If $( U,\xi _{U}) \in [ 
\mathcal{A}_{M}] $ and $( V,\eta _{V}) \in [ \mathcal{A}_{N}] $ such that $%
U\subseteq h^{-1}( V) ,$ then the modules morphism $\Gamma ( Th\circ \rho
,h\circ \eta ) $ is given locally by the equality 
\begin{equation}  \label{15}
\begin{array}{c}
\Gamma ( Th\circ \rho ,h\circ \eta ) ( z^{\alpha }t_{\alpha }) ( f_{\mid V})
=z^{\alpha }\rho _{\alpha }^{i}( \frac{\partial ( f\circ h_{\mid U}) }{%
\partial x^{i}}\circ h^{-1}) _{\mid V},%
\end{array}%
\end{equation}%
for any $z^{\alpha }t_{\alpha }\in \Gamma (F_{\mid V},\nu ,V)$ and $f\in 
\mathcal{F}( N) .$
\end{proposition}

\begin{proof}
Since $f\in \mathcal{F}(N)$, then $f\circ h\in \mathcal{F}(M)$
and we have%
\begin{equation*}
\begin{array}{c}
\frac{\partial f\circ h}{\partial x^{i}}=\frac{\partial h^{\tilde{\imath}}}{%
\partial x^{i}}\cdot \frac{\partial f}{\partial \varkappa ^{\tilde{\imath}}}%
\circ h \\ 
~\ \ \ \ \ \ \ \ \ \ \ \ \ \ \ \ \ \ \ \ \ \ \ \ \ \ \ \ \Updownarrow
\left\Vert A_{\tilde{\imath}}^{i}\right\Vert =\left\Vert \frac{\partial h^{%
\tilde{\imath}}}{\partial x^{i}}\right\Vert ^{-1} \\ 
\frac{\partial f}{\partial \varkappa ^{\tilde{\imath}}}\circ h=A_{\tilde{%
\imath}}^{i}\cdot \frac{\partial f\circ h}{\partial x^{i}} \\ 
\Updownarrow \\ 
\frac{\partial f}{\partial \varkappa ^{\tilde{\imath}}}=(A_{\tilde{\imath}%
}^{i}\circ h^{-1})(\frac{\partial f\circ h}{\partial x^{i}}\circ h^{-1}).%
\end{array}%
\end{equation*}%
We now consider that 
\begin{equation*}
\begin{array}{c}
\Gamma (Th\circ \rho ,h\circ \eta )(z^{\alpha }t_{\alpha })(f_{\mid
V})=(z^{\alpha }\theta _{\alpha }^{\tilde{\imath}}\frac{\partial f}{\partial
\varkappa ^{\tilde{\imath}}}),%
\end{array}%
\end{equation*}%
where the components $\theta _{\alpha }^{\tilde{\imath}}$ are given by the
equation $\theta _{\alpha }^{\tilde{\imath}}(A_{\tilde{\imath}}^{i}\circ h^{-1})=\rho_{\alpha }^{i}$, which completes the proof.
\end{proof}
\begin{definition}
A generalized Lie algebroid is a vector bundle $(F,\nu ,N)$ given by the
diagram: 
\begin{equation}\label{diag000}
\begin{array}{c}
\begin{array}[b]{ccccc}
( F,\left[ ,\right] _{F,h}) & ^{\underrightarrow{~\ \ \ \rho \ \
\ \ }} & \left( TM,\left[ ,\right] _{TM}\right) & ^{\underrightarrow{~\ \ \
Th\ \ \ \ }} & \left( TN,\left[ ,\right] _{TN}\right) \\ 
~\downarrow \nu &  & ~\ \ \downarrow \tau _{M} &  & ~\ \ \downarrow \tau _{N}
\\ 
N & ^{\underrightarrow{~\ \ \ \eta ~\ \ }} & M & ^{\underrightarrow{~\ \ \
h~\ \ }} & N%
\end{array}
\end{array}
,
\end{equation}%
where $h$ and $\eta $ are arbitrary diffeomorphisms, $(\rho ,\eta )$ is a
vector bundles morphism from $(F,\nu ,N)$ to $(TM,\tau _{M},M)$ and 
\begin{equation*}
\begin{array}{ccc}
\Gamma \left( F,\nu ,N\right) \times \Gamma \left( F,\nu ,N\right) & ^{%
\underrightarrow{~\ \ \left[ ,\right] _{F,h}~\ \ }} & \Gamma \left( F,\nu
,N\right)\\ 
\left( u,v\right) & \longmapsto & \ \left[ u,v\right] _{F,h}
\end{array}
,
\end{equation*}%
is an operation satisfies in 
\begin{equation*}
\begin{array}{c}
\left[ u,f\cdot v\right] _{F,h}=f\left[ u,v\right] _{F,h}+\Gamma \left(
Th\circ \rho ,h\circ \eta \right) \left( u\right) f\cdot v,\ \ \ \forall
f\in \mathcal{F}(N),%
\end{array}%
\end{equation*}%
such that the 4-tuple $(\Gamma (F,\nu ,N),[,]_{F,h})$ is a Lie $%
\mathcal{F}(N)$-algebra.
\end{definition}
We denote by $((F,\nu ,N),[,]_{F,h},(\rho ,\eta ))$ the generalized
Lie algebroid defined in the above. Moreover, the couple $(\lbrack
,]_{F,h},(\rho ,\eta ))$ is called the \emph{generalized Lie algebroid
structure}.
\begin{remark}
Note that $((F,\nu ,N),[,]_{F,h},(\rho ,\eta ))$ is a generalized Lie
algebroid if and only if $$(\Gamma (F,\nu ,N),[,]_{F,h},\Gamma
(Th\circ \rho ,h\circ \eta )),$$ is a generalized Lie $\mathcal{F}(N)$%
-algebra. This algebra will be called the generalized Lie $\mathcal{F}(N)$%
-algebra of the generalized Lie algebroid $((F,\nu ,N),[,]_{F,h},(\rho
,\eta ))$.
\end{remark}
\begin{proposition}
Let $(F,\nu ,N)$ be a vector bundle given by the diagram (\ref{diag000}). If for any $%
(V,\eta _{V})\in \lbrack \mathcal{A}_{N}]$ we have a biadditive operation 
\begin{equation*}
\begin{array}{ccc}
\Gamma (F_{\mid V},\nu ,V)\times \Gamma (F_{\mid V},\nu ,V) & ^{%
\underrightarrow{~\ \ [,]_{F_{\mid V,h}}~\ \ }} & \Gamma (F_{\mid V},\nu ,V)\\ 
(u,v) & \longmapsto & \ [u,v]_{F_{\mid V,h}}
\end{array}
,
\end{equation*}
such that the 4-tuple $(\Gamma (F_{\mid V},\nu ,V),[,]_{F_{\mid
V,h}},\Gamma (Th\circ \rho ,h\circ \eta ))$ is a generalized Lie $\mathcal{F}%
(N)_{\mid V}$-algebra, then $((F,\nu ,N),[,]_{F,h},(\rho ,\eta ))$ is a
generalized Lie algebroid, where $[,]_{F,h}$ is the biadditive operation
given by%
\begin{equation*}
\begin{array}{ccc}
\Gamma (F,\nu ,N)\times \Gamma (F,\nu ,N) & ^{\underrightarrow{~\ \
[,]_{F,h}~\ \ }} & \Gamma (F,\nu ,N) \\ 
(u,v) & \longmapsto & \ \ [u,v]_{F,h}%
\end{array}
,
\end{equation*}%
such that $[u,v]_{F,h\mid V}$ $=[u_{\mid V},v_{\mid V}]_{F_{\mid V},h},$ for
any $(V,\eta _{V})\in \lbrack \mathcal{A}_{N}].$
\end{proposition}
Now suppose that $\{t_{\alpha }\}$ is a local basis for the $\mathcal{F}( N) _{\mid V}$%
-module of sections of $(F_{\mid V},\nu ,V)$ and we put $[t_{\alpha
},t_{\beta }]_{F_{\mid V},h}=L_{\alpha \beta }^{\gamma }t_{\gamma }$, where $%
\alpha ,\beta ,\gamma \in \{1,\ldots ,p\}$. It is easy to see that $%
L_{\alpha \beta }^{\gamma }=-L_{\beta \alpha }^{\gamma }$. According to
Proposition \ref{Prop18}, $\Gamma(Th\circ\rho, h\circ\eta)$ is a Lie
algebras morphism ans so we obtain 
\begin{equation*}
\begin{array}{c}
\displaystyle(L_{\alpha \beta }^{\gamma }\circ h)(\rho _{\gamma }^{k}\circ
h)=(\rho _{\alpha }^{i}\circ h)\frac{\partial (\rho _{\beta }^{k}\circ h)}{%
\partial x^{i}}-(\rho _{\beta }^{j}\circ h)\frac{\partial (\rho _{\alpha
}^{k}\circ h)}{\partial x^{j}}.%
\end{array}%
\end{equation*}
The local real-valued functions $L_{\alpha \beta }^{\gamma }$ introduced in
the above are called the {\textit{structure functions}} of the generalized
Lie algebroid $((F,\nu ,N),[,]_{F,h},(\rho ,\eta ))$ (see \cite{A}-\cite{A4} for more details).

A morphism from $((F,\nu ,N),[,]_{F,h},(\rho ,\eta ))$ to $((F^{\prime },\nu
^{\prime },N^{\prime }),[,]_{F^{\prime },h^{\prime }},(\rho ^{\prime },\eta
^{\prime }))$ is a vector bundles morphism $(\varphi ,\varphi _{0})$ from $%
(F,\nu ,N)$ to $(F^{\prime },\nu ^{\prime },N^{\prime })$ such that $\varphi
_{0}$ is a diffeomorphism from $N$ to $N^{\prime }$ and the modules
morphism $\Gamma (\varphi ,\varphi _{0})$ is a Lie $\mathcal{F}(N)$-algebras
morphism from $(\Gamma (F,\nu ,N),[,]_{F,h})$ to \\$(\Gamma
(F^{\prime },\nu ^{\prime },N^{\prime }),[,]_{F^{\prime },h^{\prime
}}).$ Thus, we can discuss about the category of generalized Lie algebroids. 

Here, we present some examples of (generalized) Lie algebroids.

\begin{example}
Any Lie algebroid can be regarded as a generalized Lie algebroid.
\end{example}

\begin{example}
If $((F,\nu ,N),[,]_{F,h},(\rho ,\eta ))$ is a generalized Lie algebroid,
then we consider the pull-back vector bundle $( h^{\ast }F,h^{\ast }\nu ,M)
$. We denote by 
$( t_{\alpha }) _{\alpha \in \overline{1,n}}$ the natural base of the
generalized Lie $\mathcal{F}( N) _{\mid V}$-algebra $( \Gamma ( F_{\mid
V},\nu ,V) ,[,]_{F_{\mid V},h},\Gamma ( Th\circ \rho ,h\circ \eta )
)$, where $( V,\eta _{V}) \in [ \mathcal{A}_{N}] $. Also, we denote by $( \frac{%
\partial }{\partial x^{i}}) _{i\in \overline{1,m}}$ the natural base of the
generalized Lie $\mathcal{F}( M) _{\mid U}$-algebra $( \Gamma ( TM_{\mid
U},\tau _{M},U) ,[,]_{TM_{\mid U}},\Gamma ( Id_{TM},Id_{M}) )$, where $( U,\xi _{U}) \in [ \mathcal{A}_{M}] $. Let 
$(\overset{h^{\ast }F}{\rho },Id_{M})$ be the vector bundles morphism from $%
( h^{\ast }F,h^{\ast }\nu ,M) $ to $( TM,\tau _{M},M) $ locally given by%
\begin{equation*}
\begin{array}{rcl}
h^{\ast }F_{\mid U} & ^{\underrightarrow{\overset{h^{\ast }F}{\rho }}} & TM
\\ 
\displaystyle Z^{\alpha }T_{\alpha }( x) & \longmapsto & \displaystyle(
Z^{\alpha }\cdot \rho _{\alpha }^{i}\circ h) \frac{\partial }{\partial x^{i}}%
( x)%
\end{array}
.
\end{equation*}
We consider the biadditive operation 
\begin{equation*}
\begin{array}{ccc}
\Gamma ( h^{\ast }F_{\mid U},h^{\ast }\nu ,U) \times \Gamma ( h^{\ast
}F_{\mid U},h^{\ast }\nu ,U) & ^{\underrightarrow{~\ \ [ ,] _{h^{\ast
}F_{\mid U}}~\ \ }} & \Gamma ( h^{\ast }F_{\mid U},h^{\ast }\nu ,U)%
\end{array}
,
\end{equation*}
defined by 
\begin{align*}
&[T_{\alpha },f_{\mid U}\cdot T_{\beta }]_{h^{\ast }F_{\mid U}} =f_{\mid
U}\cdot ( L_{\alpha \beta }^{\gamma }\circ h) _{\mid U}T_{\gamma }+( \rho
_{\alpha }^{i}\circ h) _{\mid U}\frac{\partial f_{\mid U}}{\partial x^{i}}%
T_{\beta }, \\
&[f_{\mid U}\cdot T_{\alpha },T_{\beta }]_{h^{\ast }F_{\mid U}}=-[T_{\beta
},f_{\mid U}\cdot T_{\alpha }]_{h^{\ast }F_{\mid U}},
\end{align*}%
for any $f\in \mathcal{F}(M)$. After some calculations, we derive that the
4-tuple 
\begin{equation*}
(\Gamma ( h^{\ast }F_{\mid U},h^{\ast }\nu ,U) ,[ ,] _{h^{\ast
}F_{\mid U}},\Gamma ( \overset{h^{\ast }F}{\rho },Id_{M})),
\end{equation*}
is a generalized Lie $\mathcal{F}(M)_{\mid U}$-algebra. So, using the
biadditive operation
\begin{equation*}
\begin{array}{ccc}
\Gamma ( h^{\ast }F,h^{\ast }\nu ,M) \times \Gamma ( h^{\ast }F,h^{\ast }\nu
,M) & ^{\underrightarrow{~\ \ [ ,] _{h^{\ast }F}~\ \ }} & \Gamma ( h^{\ast
}F,h^{\ast }\nu ,M) \\ 
( T,Z) & \longmapsto & \ [ T,Z] _{h^{\ast }F}%
\end{array}
,
\end{equation*}%
given by $\ [ T,Z] _{h^{\ast }F\mid U}$ $=[ T_{\mid U},Z_{\mid U}]
_{h^{\ast }F_{\mid U}},$ for any $( U,\xi _{U}) \in [ \mathcal{A}_{M}] ,$ it
results that $$( ( h^{\ast }F,h^{\ast }\nu ,M) ,[ ,] _{h^{\ast }F},( \overset{%
h^{\ast }F}{\rho },Id_{M}) ),$$ is a Lie algebroid which is called the
pull-back Lie algebroid of the generalized Lie algebroid $$( ( F,\nu ,N) ,[
,] _{F,h},( \rho ,\eta ) ).$$
\end{example}
\begin{example}
We consider the vector bundle $( TN,\tau _{N},N) 
$ given by the diagram: 
\begin{equation*}
\begin{array}{c}
\begin{array}[b]{ccccc}
TN & ^{\underrightarrow{~\ \ \ T\eta \ \ \ \ }} & TM & ^{\underrightarrow{~\
\ \ Th\ \ \ \ }} & ( TN,[ ,] _{TN}) \\ 
~\ \ \downarrow \tau _{N} &  & ~\ \ \downarrow \tau _{M} &  & ~\ \
\downarrow \tau _{N} \\ 
N & ^{\underrightarrow{~\ \ \ \eta ~\ \ }} & M & ^{\underrightarrow{~\ \ \
h~\ \ }} & N%
\end{array}%
\end{array}
,
\end{equation*}%
where $M$ and $N$ are two manifolds and $h$ and $\eta $ are two diffeomorphisms. We denote by $(
t_{\alpha }) _{\alpha \in \overline{1,n}}$ a base of the $\mathcal{F}( N)
_{\mid V}$-module $ \Gamma ( TN_{\mid V},\tau _{N},V)$, where $(
V,\eta _{V}) \in [ \mathcal{A}_{N}]$. Now, we consider the biadditive
operation\textrm{\ }%
\begin{equation*}
\begin{array}{ccc}
\Gamma ( TN_{\mid V},\tau _{N},V) \times \Gamma ( TN_{\mid V},\tau _{N},V) & 
^{\underrightarrow{~\ \ [ ,] _{TN_{\mid V},h}~\ \ }} & ( TN_{\mid V},\tau
_{N},V) \\ 
( u,v) & \longmapsto & \ [ u,v] _{TN_{\mid V},h}%
\end{array}
,
\end{equation*}%
defined by%
\begin{align*}
[ t_{\alpha },f_{\mid V}\cdot t_{\beta }] _{TN_{\mid V},h} & =\tilde{\theta}%
( [\theta (t_{\alpha }), \theta( f_{\mid V}\cdot t_{\beta })] _{TN_{\mid
V}}),\\ 
[ f_{\mid V}\cdot t_{\alpha },t_{\beta }] _{TN_{\mid V},h} & =- [ t_{\beta
},f_{\mid V}\cdot t_{\alpha }] _{TN_{\mid V},h},
\end{align*}
where $\tilde{\theta}:=\Gamma ( T( h \circ\eta) ^{-1},( h\circ\eta) ^{-1})$
and $\theta:=\Gamma ( Th\circ T\eta ,h\circ \eta )$. We denoted by $( \frac{\partial }{%
\partial x^{i}}) _{i\in \overline{1,n}}$ the natural base for the $\mathcal{F%
}( M) _{\mid U}$-module $ \Gamma ( TM_{\mid U},\tau _{M},U)$, where $( U,\xi _{U}) \in
[ \mathcal{A}_{M}] $. If 
$[ t_{\alpha },t_{\beta }] _{TN_{\mid V},h}=L_{\alpha \beta }^{\gamma
}t_{\gamma },$ then 
\begin{equation*}
\begin{array}{c}
L_{\alpha \beta }^{\gamma }=\tilde{\theta}_{\tilde{j}}^{\gamma }( \theta
_{\alpha }^{\tilde{\imath}}\frac{\partial \theta _{\beta }^{\tilde{j}}}{%
\partial \varkappa ^{\tilde{\imath}}}-\theta _{\beta }^{\tilde{\imath}}\frac{%
\partial \theta _{\alpha }^{\tilde{j}}}{\partial x^{\tilde{\imath}}})
,~\alpha ,\beta ,\gamma \in \overline{1,n},%
\end{array}%
\end{equation*}%
where $\theta _{\alpha }^{\tilde{\imath}},~\tilde{\imath},\alpha \in 
\overline{1,n}$, are real local functions such that $\theta (t_{\alpha })
=\theta _{\alpha }^{\tilde{\imath}}\frac{\partial }{\partial x^{\tilde{\imath%
}}}$ and $\tilde{\theta}_{\tilde{j}}^{\gamma },~\tilde{j},\gamma \in 
\overline{1,n} $ are real local functions such that $\tilde{\theta}( \frac{%
\partial }{\partial x^{\tilde{j}}}) =\tilde{\theta}_{\tilde{j}}^{\gamma
}t_{\gamma } $. Obvious that $\tilde{\theta}(\theta(t_\alpha))=t_\alpha$ and 
$\tilde{\theta}^{\gamma}_{\tilde{j}}\theta_{\gamma}^{\tilde{i}}=\delta_{%
\tilde{j}}^{\tilde{i}}$. For any $u\in \Gamma ( TN,\tau _{N},N) ,$ we obtain 
\begin{equation*}
[ u_{\mid V},u_{\mid V}] _{TN_{\mid V},h}=\tilde{\theta}([\theta(u_{\mid
V}), \theta(u_{\mid V})] _{TN_{\mid V}})=\tilde{\theta}(0)=0.
\end{equation*}
Similarly, for any $u,v\in \Gamma ( TN,\tau _{N},N) $ and $f\in \mathcal{F}(N)$, we get
\begin{equation*}
\begin{array}{l}
[ u_{\mid V},f_{\mid V}\cdot v_{\mid V}] _{TN_{\mid V},h}=\tilde{\theta}(
[\theta(u_{\mid V}), \theta(f_{\mid V}\cdot v_{\mid V}) ] _{TN_{\mid V}})=%
\tilde{\theta}( [ \theta(u_{\mid V}),f_{\mid V}\cdot \theta(v_{\mid V})]
_{TN_{\mid V}}) \\ 
=\tilde{\theta}( f_{\mid V}\cdot [ \theta (u_{\mid V}),\theta(v_{\mid V})]
_{TN_{\mid V}}) +\tilde{\theta}(\theta(u_{\mid V}) ( f_{\mid V})
\cdot\theta(v_{\mid V})) \\ 
=f_{\mid V}\cdot\tilde{\theta}([\theta(u_{\mid V}), \theta(v_{\mid V})]
_{TN_{\mid V}})+\theta(u_{\mid V}) (f_{\mid V}) \cdot \tilde{\theta}%
(\theta(v_{\mid V})) \\ 
=f_{\mid V}\cdot [ u_{\mid V},v_{\mid V}] _{TN_{\mid V},h}+ \theta(u_{\mid
V})( f_{\mid V})\cdot v_{\mid V}.%
\end{array}%
\end{equation*}
Relation (\ref{15}) gives us 
\begin{equation*}
\begin{array}{c}
[ z_{\mid V},f_{\mid V}\cdot v_{\mid V}] _{TN_{\mid V},h}=f_{\mid V}\cdot [
z_{\mid V},v_{\mid V}] _{TN_{\mid V},h}+( z^{\alpha }\rho _{\alpha }^{i}( 
\frac{\partial f\circ h}{\partial x^{i}}\circ h^{-1}) _{\mid V}) \cdot
v_{\mid V},%
\end{array}%
\end{equation*}%
for any $z,v\in \Gamma (TN,\tau _{N},N) $ and $f\in \mathcal{F}(N)$. Also, relation
\begin{equation*}
\theta[ u_{\mid V},v_{\mid V}] _{TN_{\mid V},h}=\theta\{ \tilde{\theta}( [
\theta( u_{\mid V}),\theta(v_{\mid V})]_{TN_{\mid V}}) \} =[\theta(u_{\mid
V}), \theta(v_{\mid V})]_{TN_{\mid V}},
\end{equation*}
implies that 
\begin{equation*}
[u_{\mid V},[v_{\mid V},z_{\mid V}] _{TN_{\mid V},h}] _{TN_{\mid V},h}=%
\tilde{\theta}([\theta(u_{\mid V}), [\theta(v_{\mid V}), \theta(z_{\mid
V})]_{TN_{\mid V}}] _{TN_{\mid V}}),
\end{equation*}%
for any $u,v,z\in \Gamma ( TN,\tau _{N},N)$. Now, using the Jacobi identity
for the generalized Lie $\mathcal{F}( N)_{\mid V}$-algebra $( \Gamma (
TN_{\mid V},\tau _{N},V) ,[ ,] _{TN_{\mid V}}) ,$ we obtain%
\begin{align*}
&[\theta(u_{\mid V}), [\theta(v_{\mid V}), \theta(z_{\mid V})] _{TN_{\mid
V}}] _{TN_{\mid V}}+[\theta(v_{\mid V}), [\theta(z_{\mid V}), \theta(u_{\mid
V})] _{TN_{\mid V}}] _{TN_{\mid V}} \\
&+[\theta(z_{\mid V}), [\theta(u_{\mid V}), \theta(v_{\mid V})] _{TN_{\mid
V}}] _{TN_{\mid V}} =0,\ \ \ \forall u,v,z\in \Gamma ( TN,\tau _{N},N) .
\end{align*}
Two last equations give us 
\begin{align*}
&[ u_{\mid V},[ v_{\mid V},z_{\mid V}] _{TN_{\mid V},h}] _{TN_{\mid V},h}+[
z_{\mid V},[ u_{\mid V},v_{\mid V}] _{TN_{\mid V},h}] _{TN_{\mid V},h} \\
&+[ v_{\mid V},[ z_{\mid V},u_{\mid V}] _{TN_{\mid V},h}] _{TN_{\mid V},h}=0,
\end{align*}%
and so Jacobi identity holds for $%
[,]_{TN_{\mid V},h}$. Thus, the 4-tuple $$(\Gamma ( TN_{\mid V},\tau
_{N},V) ,[,]_{TN_{\mid V},h},\Gamma ( Th\circ T\eta ,h\circ \eta )
),$$ is a generalized Lie $\mathcal{F}(N)_{\mid V}$-algebra. Therefore, using the
biadditive operation 
\begin{equation*}
\begin{array}{ccc}
\Gamma ( TN,\nu ,N) \times \Gamma ( TN,\nu ,N) & ^{\underrightarrow{~\ \ [
,] _{TN,h}~\ \ }} & \Gamma ( TN,\nu ,N) \\ 
( u,v) & \longmapsto & \ [ u,v] _{TN,h}%
\end{array}
,
\end{equation*}%
given by $[ u,v] _{TN,h\mid V}$ $=[ u_{\mid V},v_{\mid V}] _{TN_{\mid
V},h}, $ for any $( V,\eta _{V}) \in [ \mathcal{A}_{N}] ,$ it results that $$
( ( TN,\tau _{N},N) ,[ ,] _{TN,h},( T\eta ,\eta ) ),$$ is a generalized Lie
algebroid.
\end{example}

For any difeomorphisms $\eta $ and $h,$ new and interesting generalized Lie
algebroids structures for the tangent vector bundle $( TN,\tau _{N},N) $ are
obtained$.$ In particular, using arbitrary isometries (symmetries,
translations, rotations,...) for the Euclidean $3$-dimensional space $\Sigma
,$ and arbitrary basis for the module of sections we obtain a lot of
generalized Lie algebroids structures for the tangent vector bundle $(
T\Sigma ,\tau _{\Sigma },\Sigma ) $.

\begin{example}
Let $((F,\nu ,N),[,]_{F},(\rho ,Id_{N}))$ be a Lie algebroid, $(U,\xi
_{U})\in \lbrack \mathcal{A}_{N}]$ and $(V,\eta _{V})\in \lbrack \mathcal{A}%
_{N}]$ such that $U\subseteq h^{-1}(V)$. We denote by $(t_{\alpha })_{\alpha
\in \overline{1,n}}$ the natural base of the generalized Lie $\mathcal{F}%
(N)_{\mid V}$-algebra $(\Gamma (F_{\mid V},\nu ,V),[,]_{F_{\mid
V}},\Gamma (\rho ,Id_{N}))$. If $h\in Diff(N),$ then we consider the
biadditive operation\textrm{\ }%
\begin{equation*}
\begin{array}{ccc}
\Gamma (F_{\mid V},\nu ,V)\times \Gamma (F_{\mid V},\nu ,V) & ^{%
\underrightarrow{~\ \ [,]_{F_{\mid V},h}~\ \ }} & \Gamma (F_{\mid V},\nu ,V)
\\ 
(u_{\mid V},v_{\mid V}) & \longmapsto & \ [u_{\mid V},v_{\mid V}]_{F_{\mid
V},h}%
\end{array}%
,
\end{equation*}%
defined by%
\begin{equation}
\begin{array}{c}
\lbrack t_{\alpha },f_{\mid V}\cdot t_{\beta }]_{F_{\mid V},h}=f_{\mid
V}\cdot \lbrack t_{\alpha },t_{\beta }]_{F_{\mid V}}+\rho _{\alpha }^{i}(%
\frac{\partial f\circ h}{\partial x^{i}}\circ h^{-1})_{\mid V}\cdot t_{\beta
},%
\end{array}
\label{TM}
\end{equation}%
and 
\begin{equation*}
\lbrack f_{\mid V}\cdot t_{\alpha },t_{\beta }]_{F_{\mid V},h}=-[t_{\beta
},f_{\mid V}\cdot t_{\alpha }]_{F_{\mid V},h},
\end{equation*}
for any $f\in \mathcal{F}(N)$. Easily, we obtain  $[f_{\mid V}\cdot
t_{\alpha },f_{\mid V}\cdot t_{\alpha }]_{F_{\mid V},h}=0$ and consequently
\begin{equation*}
\begin{array}{c}
\lbrack u_{\mid V},u_{\mid V}]_{F_{\mid V},h}=0,\ \ \ \forall u\in \Gamma
(F,\nu ,N).%
\end{array}%
\end{equation*}%
Since $((F,\nu ,N),[,]_{F},(\rho ,Id_{N}))$ is a Lie algebroid, then 
\begin{equation*}
\begin{array}{c}
\lbrack t_{\alpha },[t_{\beta },t_{\gamma }]_{F_{\mid V}}]_{F_{\mid
V}}+[t_{\gamma },[t_{\alpha },t_{\beta }]_{F_{\mid V}}]_{F_{\mid
V}}+[t_{\beta },[t_{\gamma },t_{\alpha }]_{F_{\mid V}}]_{F_{\mid V}}=0,%
\end{array}%
\end{equation*}%
which gives us 
\begin{align}
& L_{\beta \gamma }^{\theta }L_{\alpha \theta }^{\mu }+\rho _{\alpha }^{i}(%
\frac{\partial L_{\beta \gamma }^{\theta }\circ h}{\partial x^{i}}\circ
h^{-1})_{\mid V}+L_{\alpha \beta }^{\theta }L_{\gamma \theta }^{\mu }+\rho
_{\gamma }^{i}(\frac{\partial L_{\alpha \beta }^{\theta }\circ h}{\partial
x^{i}}\circ h^{-1})_{\mid V}  \notag  \label{PP1} \\
& +L_{\gamma \alpha }^{\theta }L_{\beta \theta }^{\mu }+\rho _{\beta }^{i}(%
\frac{\partial L_{\gamma \alpha }^{\theta }\circ h}{\partial x^{i}}\circ
h^{-1})_{\mid V}=0,
\end{align}%
where $\rho _{\alpha }^{i}$ and $L_{\alpha \beta }^{\theta }$ are
real-valued functions defined on $U$. As%
\begin{equation*}
\begin{array}{l}
\Gamma (Th\circ \rho ,h)[t_{\alpha },t_{\beta }]_{F_{\mid V}}=\Gamma
(Th,h)[\Gamma (\rho ,Id_{N})t_{\alpha },\Gamma (\rho ,Id_{N})t_{\beta
}]_{TN_{\mid V}} \\ 
=[\Gamma (Th,h)\circ \Gamma (\rho ,Id_{N})t_{\alpha },\Gamma (Th,h)\circ
\Gamma (\rho ,Id_{N})t_{\beta }]_{TN_{\mid V}} \\ 
=[\Gamma (Th\circ \rho ,h)t_{\alpha },\Gamma (Th\circ \rho ,h)t_{\beta
}]_{TN_{\mid V}},%
\end{array}%
\end{equation*}%
we get
\begin{equation*}
\begin{array}{c}
L_{\alpha \beta }^{\gamma }\rho _{\gamma }^{k}(\frac{\partial h^{\tilde{%
\imath}}}{\partial x^{k}}\circ h^{-1})_{\mid V}\frac{\partial }{\partial
\varkappa ^{\tilde{\imath}}}=(\theta _{\alpha }^{\tilde{j}}\frac{\partial
\theta _{\beta }^{\tilde{\imath}}}{\partial \varkappa ^{\tilde{j}}}-\theta
_{\beta }^{\tilde{j}}\frac{\partial \theta _{\alpha }^{\tilde{\imath}}}{%
\partial \varkappa ^{\tilde{j}}})\frac{\partial }{\partial \varkappa ^{%
\tilde{\imath}}} \\ 
\Updownarrow \\ 
L_{\alpha \beta }^{\gamma }\rho _{\gamma }^{k}(\frac{\partial h^{\tilde{%
\imath}}}{\partial x^{k}}\circ h^{-1})_{\mid V}\frac{\partial }{\partial
\varkappa ^{\tilde{\imath}}}=(\theta _{\alpha }^{\tilde{j}}\frac{\partial
\lbrack \rho _{\beta }^{k}(\frac{\partial h^{\tilde{\imath}}}{\partial x^{k}}%
\circ h^{-1})_{\mid V}]}{\partial \varkappa ^{\tilde{j}}}-\theta _{\beta }^{%
\tilde{j}}\frac{\partial \lbrack \rho _{\alpha }^{k}(\frac{\partial h^{%
\tilde{\imath}}}{\partial x^{k}}\circ h^{-1})_{\mid V}]}{\partial \varkappa
^{\tilde{j}}})\frac{\partial }{\partial \varkappa ^{\tilde{\imath}}} \\ 
\Updownarrow \\ 
L_{\alpha \beta }^{\gamma }\rho _{\gamma }^{k}(\frac{\partial h^{\tilde{%
\imath}}}{\partial x^{k}}\circ h^{-1})_{\mid V}\frac{\partial }{\partial
\varkappa ^{\tilde{\imath}}}=(\theta _{\alpha }^{\tilde{j}}\frac{\partial
\rho _{\beta }^{k}}{\partial \varkappa ^{\tilde{j}}}-\theta _{\beta }^{%
\tilde{j}}\frac{\partial \rho _{\alpha }^{k}}{\partial \varkappa ^{\tilde{j}}%
})(\frac{\partial h^{\tilde{\imath}}}{\partial x^{k}}\circ h^{-1})_{\mid V}%
\frac{\partial }{\partial \varkappa ^{\tilde{\imath}}} \\ 
\Updownarrow \\ 
L_{\alpha \beta }^{\gamma }\rho _{\gamma }^{k}=\theta _{\alpha }^{\tilde{j}}%
\frac{\partial \rho _{\beta }^{k}}{\partial \varkappa ^{\tilde{j}}}-\theta
_{\beta }^{\tilde{j}}\frac{\partial \rho _{\alpha }^{k}}{\partial \varkappa
^{\tilde{j}}} \\ 
\Updownarrow \\ 
L_{\alpha \beta }^{\gamma }\rho _{\gamma }^{k}=\rho _{\alpha }^{j}(\frac{%
\partial \rho _{\beta }^{k}\circ h}{\partial x^{j}}\circ h^{-1})_{\mid
V}-\rho _{\beta }^{j}(\frac{\partial \rho _{\alpha }^{k}\circ h}{\partial
x^{j}}\circ h^{-1})_{\mid V} \\ 
\Updownarrow%
\end{array}%
\end{equation*}%
\begin{equation}
\begin{array}{c}
(L_{\alpha \beta }^{\gamma }\circ h)(\rho _{\gamma }^{k}\circ h)=(\rho
_{\alpha }^{j}\circ h)\frac{\partial \rho _{\beta }^{k}\circ h}{\partial
x^{j}}-(\rho _{\beta }^{j}\circ h)\frac{\partial \rho _{\alpha }^{k}\circ h}{%
\partial x^{j}}.%
\end{array}
\label{PP2}
\end{equation}%
Using (\ref{TM}) we get 
\begin{align*}
& [t_{\alpha },[t_{\beta },f_{\mid V}t_{\gamma }]_{F_{\mid V},h}]_{F_{\mid
V},h}=f_{\mid V}L_{\beta \gamma }^{\theta }L_{\alpha \theta }^{\mu }t_{\mu
}+\rho _{\alpha }^{i}(\frac{\partial f\circ h}{\partial x^{i}}\circ
h^{-1})_{\mid V}L_{\beta \gamma }^{\theta }t_{\theta } \\
& +f_{\mid V}\rho _{\alpha }^{i}(\frac{\partial L_{\beta \gamma }^{\theta
}\circ h}{\partial x^{i}}\circ h^{-1})_{\mid V}t_{\theta }+\rho _{\beta
}^{j}(\frac{\partial f\circ h}{\partial x^{j}}\circ h^{-1})_{\mid
V}L_{\alpha \gamma }^{\theta }t_{\theta } \\
& +\rho _{\alpha }^{i}(\frac{\partial \rho _{\beta }^{j}\circ h}{\partial
x^{i}}\circ h^{-1})_{\mid V}(\frac{\partial f\circ h}{\partial x^{j}}\circ
h^{-1})_{\mid V}t_{\gamma }+\rho _{\beta }^{j}\rho _{\alpha }^{i}\frac{%
\partial }{\partial x^{i}}(\frac{\partial f\circ h}{\partial x^{j}}\circ
h^{-1})_{\mid V}t_{\gamma },
\end{align*}%
\begin{align*}
& [f_{\mid V}t_{\gamma },[t_{\alpha },t_{\beta }]_{F_{\mid V},h}]_{F_{\mid
V},h}=f_{\mid V}L_{\alpha \beta }^{\theta }L_{\gamma \theta }^{\mu }t_{\mu
}-L_{\alpha \beta }^{\theta }\rho _{\theta }^{i}(\frac{\partial f\circ h}{%
\partial x^{i}}\circ h^{-1})_{\mid V}t_{\gamma } \\
& +f_{\mid V}\rho _{\gamma }^{i}(\frac{\partial L_{\alpha \beta }^{\theta
}\circ h}{\partial x^{i}}\circ h^{-1})_{\mid V}t_{\theta },
\end{align*}%
and 
\begin{align*}
& [t_{\beta },[f_{\mid V}t_{\gamma },t_{\alpha }]_{F_{\mid V},h}]_{F_{\mid
V},h}=f_{\mid V}L_{\gamma \alpha }^{\theta }L_{\beta \theta }^{\mu }t_{\mu
}-\rho _{\beta }^{i}(\frac{\partial f\circ h}{\partial x^{i}}\circ
h^{-1})_{\mid V}L_{\alpha \gamma }^{\theta }t_{\theta } \\
& -\rho _{\alpha }^{j}(\frac{\partial f\circ h}{\partial x^{j}}\circ
h^{-1})_{\mid V}L_{\beta \gamma }^{\theta }t_{\theta }-\rho _{\beta }^{i}(%
\frac{\partial \rho _{\alpha }^{j}\circ h}{\partial x^{i}}\circ
h^{-1})_{\mid V}(\frac{\partial f\circ h}{\partial x^{j}}\circ h^{-1})_{\mid
V}t_{\gamma } \\
& -\rho _{\alpha }^{i}\rho _{\beta }^{j}\frac{\partial }{\partial x^{j}}(%
\frac{\partial f\circ h}{\partial x^{i}}\circ h^{-1})_{\mid V}t_{\gamma
}+f_{\mid V}\rho _{\beta }^{i}(\frac{\partial L_{\gamma \alpha }^{\theta
}\circ h}{\partial x^{i}}\circ h^{-1})_{\mid V}t_{\theta }.
\end{align*}%
Summing above three equations and using (\ref{PP1}) and (\ref{PP2}) imply
that 
\begin{align*}
& [t_{\alpha },[t_{\beta },f_{\mid V}t_{\gamma }]_{F_{\mid V},h}]_{F_{\mid
V},h}+[f_{\mid V}t_{\gamma },[t_{\alpha },t_{\beta }]_{F_{\mid
V},h}]_{F_{\mid V},h} \\
& +[t_{\beta },[f_{\mid V}t_{\gamma },t_{\alpha }]_{F_{\mid V},h}]_{F_{\mid
V},h}=0.
\end{align*}%
So, we obtain 
\begin{align*}
& [u_{\mid V},[v_{\mid V},z_{\mid V}]_{F_{\mid V},h}]_{F_{\mid
V},h}+[z_{\mid V},[u_{\mid V},v_{\mid V}]_{F_{\mid V},h}]_{F_{\mid V},h} \\
& +[v_{\mid V},[z_{\mid V},u_{\mid V}]_{F_{\mid V},h}]_{F_{\mid V},h}=0,
\end{align*}%
for any $u,v,z\in \Gamma (F,\nu ,N).$ Therefore, the 4-tuple $(\Gamma
(F_{\mid V},\nu ,V),[,]_{F_{\mid V},h})$ is a generalized Lie $%
\mathcal{F}(N)_{\mid V}$-algebra. Moreover, using the biadditive operation given by%
\begin{equation*}
\begin{array}{ccc}
\Gamma (F,\nu ,N)\times \Gamma (F,\nu ,N) & ^{\underrightarrow{~\ \
[,]_{F,h}~\ \ }} & \Gamma (F,\nu ,N) \\ 
(u,v) & \longmapsto & \ [u,v]_{F,h}%
\end{array}%
,
\end{equation*}%
given by $[u,v]_{F,h\mid V}$ $=[u_{\mid V},v_{\mid V}]_{F_{\mid V},h},$ for
any $(V,\eta _{V})\in \lbrack \mathcal{A}_{N}]$, we derive that $((F,\nu
,N),[,]_{F,h},(T\eta \circ \rho ,\eta ))$ is a generalized Lie algebroid.
\end{example}
We remark that for any Lie algebroid $( (F,\nu ,N),[ ,] _{F},( \rho ,Id_{N})
) $ and for any $h\in Diff( N) $ we obtain a new generalized Lie algebroid $%
( (F,\nu ,N),[ ,] _{F,h},( \rho ,Id_{N}) )$.

In the future we will use the same notations for the global Lie bracket and
local Lie bracket, but we understand the difference with respect to the
context. 

\section{Exterior differential calculus}

In the next we consider $\mathcal{F}$ as a commutative ring. In this section, we propose an exterior differential calculus in the general
framework of a generalized Lie algebras. Note that for any $q\in \mathbb{N}%
^* $ we denote by $(\Sigma _{q},\circ ) $ the permutations group of the set $%
\{ 1,2,\cdots ,q\}$.

\begin{definition}
Let $( A,[ ,] _{A},\rho )$ be a generalized Lie algebra. A $q$-linear operation 
\begin{equation*}
\begin{array}{ccc}
A^{q} & ^{\underrightarrow{\ \ \omega \ \ }} & \mathcal{F} \\ 
( z_{1},\cdots,z_{q}) & \longmapsto & \omega ( z_{1},\cdots,z_{q})%
\end{array}
,
\end{equation*}%
such that 
\begin{equation*}
\begin{array}{c}
\omega ( z_{\sigma ( 1) },\cdots,z_{\sigma ( q) }) =sgn( \sigma )\omega (
z_{1},\cdots,z_{q}),%
\end{array}%
\end{equation*}%
for any $z_{1},\cdots,z_{q}\in A$ and for any $\sigma \in \Sigma _{q}$, will
be called a form of degree $q$ (or $q$-form) and the set of $q$-forms will be
denoted by $\Lambda^q(A)$.
\end{definition}
Using the above definition, if $\omega \in \Lambda ^{q}( A)$, then $\omega (
z_{1},\cdots,z,\cdots,z,\cdots,z_{q}) =0.$ Therefore, if $\omega \in \Lambda
^{q}( A) $, then 
\begin{equation*}
\omega ( z_{1},\cdots,z_{i},\cdots,z_{j},\cdots, z_{q}) =-\omega (
z_{1},\cdots,z_{j},\cdots,z_{i},\cdots,z_{q}) .
\end{equation*}

\begin{proposition}
If $q\in \mathbb{N}$, then $\Lambda ^{q}( A) $ is an $%
\mathcal{F}$-module.
\end{proposition}
\begin{definition}
If $\omega \in \Lambda ^{q}( A) $ and $\theta \in \Lambda ^{r}( A) $, then
the $( q+r) $-form $\omega \wedge \theta $ defined by%
\begin{align}  \label{form1}
\omega \wedge \theta ( z_{1},\cdots ,z_{q+r})&=\hspace{-.8cm} \underset{%
\underset{\sigma ( q+1) <\cdots <\sigma ( q+r) }{ {\sigma ( 1) <\cdots
<\sigma ( q)} }}{\sum }sgn( \sigma ) \omega ( z_{\sigma ( 1) },\cdots
,z_{\sigma ( q) }) \cdot\theta ( z_{\sigma ( q+1) },\cdots ,z_{\sigma ( q+r)
})  \notag \\
&\!\!\!=\!\!\frac{1}{q!r!}\underset{\sigma \in \Sigma _{q+r}}{\sum }sgn(
\sigma ) \omega ( z_{\sigma ( 1) },\cdots ,z_{\sigma ( q) })\cdot\theta (
z_{\sigma ( q+1) },\cdots ,z_{\sigma ( q+r) }) ,
\end{align}
for any $z_{1},\cdots ,z_{q+r}\in A,$ will be called \emph{the exterior
product of the forms }$\omega $ \emph{and}~$\theta .$
\end{definition}
Using the above definition, we obtain
\begin{theorem}
Let $f\in \mathcal{F}$, $\omega, \bar{\omega}\in\Lambda ^{q}(A)$, $\theta, 
\bar{\theta}\in\Lambda ^{r}(A)$ and $\eta\in\Lambda ^{s}(A)$. Then we have 
\begin{align*}
&\omega \wedge \theta=(-1)^{q\cdot r}\theta \wedge \omega,\ \ \ (f\cdot \omega ) \wedge \theta =f\cdot ( \omega \wedge \theta ) =\omega
\wedge ( f\cdot \theta ),\\
&(\omega+\bar{\omega})\wedge\theta =\omega \wedge\theta+\bar{\omega}\wedge
\theta,\ \ \ \omega \wedge ( \theta+\bar{\theta}) =\omega \wedge \theta
+\omega \wedge\bar{\theta},\\
&(\omega\wedge \theta)\wedge \eta =\omega \wedge(\theta\wedge \eta).
\end{align*}
\end{theorem}
\begin{theorem}
If we set 
\begin{equation*}
\Lambda ( A) =\underset{q\geq 0}{\oplus }\Lambda ^{q}( A) ,
\end{equation*}%
then $( \Lambda ( A) ,\wedge ) $ is an associative classical  $\mathcal{F}$-algebra. This
algebra will be called the exterior algebra of the $\mathcal{F}$-algebra $(
A,[ ,] _{A})$.
\end{theorem}

If $\{ t^{\alpha },~\alpha \in \overline{1,p}\} $ is the cobase associated
to the base $\{ t_{\alpha },~\alpha \in \overline{1,p}\} $ of the $\mathcal{F%
}$\emph{-}algebra $( A,[ ,] _{A}) $, then for any $q\in \overline{%
1,p}$ we define $C_{p}^{q}$ exterior forms of the type $t^{\alpha
_{1}}\wedge \cdots \wedge t^{\alpha _{q}}$, such that $1\leq \alpha
_{1}<\cdots <\alpha _{q}\leq p$. The set 
\begin{equation*}
\begin{array}{c}
\left\{ t^{\alpha _{1}}\wedge \cdots \wedge t^{\alpha _{q}}|1\leq \alpha
_{1}<\cdots <\alpha _{q}\leq p\right\},%
\end{array}%
\end{equation*}%
is a basis for the $\mathcal{F}$-module $ \Lambda ^{q}( A) $.
Therefore, if $\omega \in \Lambda ^{q}( A) $, then%
\begin{equation*}
\begin{array}{c}
\omega =\omega _{\alpha _{1}\cdots \alpha _{q}}t^{\alpha _{1}}\wedge \cdots
\wedge t^{\alpha _{q}}.%
\end{array}%
\end{equation*}
In particular, if $\omega $ is an exterior $p$-form, then it can be written as 
\begin{equation*}
\begin{array}{c}
\omega=f\cdot t^{1}\wedge \cdots \wedge t^{p},%
\end{array}%
\end{equation*}%
where $f\in \mathcal{F}$.

\begin{definition}
For any $z\in A$, the operator 
\begin{equation*}
\begin{array}{c}
\begin{array}{rcl}
\Lambda ( A) & ^{\underrightarrow{~\ \ L_{z}~\ \ }} & \Lambda ( A)%
\end{array}%
,%
\end{array}%
\end{equation*}%
defined by%
\begin{equation*}
\begin{array}{c}
L_{z}( f) =\rho ( z) ( f) ,~\forall f\in \mathcal{F},%
\end{array}%
\end{equation*}%
and 
\begin{align}  \label{form2}
L_{z}\omega ( z_{1},\cdots ,z_{q}) & =\rho ( z) ( \omega ( z_{1},\cdots
,z_{q}) ) -\overset{q}{\underset{i=1}{\sum }}\omega ( ( z_{1},\cdots ,[
z,z_{i}] _{A},\cdots ,z_{q}) ) ,
\end{align}
for any $\omega \in \Lambda ^{q}\mathbf{\ }( A) $ and $z_{1},\cdots
,z_{q}\in A,$ will be called the covariant Lie derivative with respect to
the element $z$.
\end{definition}

\begin{theorem}
If $z\in A,$ $\omega \in \Lambda ^{q}( A) $ and $\theta \in \Lambda ^{r}( A) 
$, then 
\begin{equation*}
\begin{array}{c}
L_{z}( \omega \wedge \theta ) =L_{z}\omega \wedge \theta +\omega \wedge
L_{z}\theta .%
\end{array}%
\end{equation*}
\end{theorem}
\begin{proof}
Let $z_{1},\cdots ,z_{q+r}\in A$. (\ref{form1}) and (\ref{form2}) give us 
\begin{align*}
&L_{z}( \omega \wedge \theta ) ( z_{1},\cdots ,z_{q+r}) =\rho ( z) ( (
\omega \wedge \theta ) ( z_{1},\cdots ,z_{q+r}) ) \\
&-\overset{q+r}{\underset{i=1}{\sum }}( \omega \wedge \theta ) ( (
z_{1},\cdots ,[ z,z_{i}] _{A},\cdots ,z_{q+r}) ) \\
&=\rho ( z) ( \underset{\sigma ( q+1) <\cdots <\sigma ( q+r) }{\underset{%
\sigma ( 1) <\cdots <\sigma ( q) }{\sum }}sgn( \sigma )\omega ( z_{\sigma (
1) },\cdots ,z_{\sigma ( q) }) \cdot \theta ( z_{\sigma ( q+1) },\cdots
,z_{\sigma ( q+r) }) ) \\
&\ \ \ -\overset{q+r}{\underset{i=1}{\sum }}( \omega \wedge \theta ) ( (
z_{1},\cdots ,[ z,z_{i}] _{A},\cdots ,z_{q+r})).
\end{align*}
Therefore we get 
\begin{align*}
&L_{z}( \omega \wedge \theta ) ( z_{1},\cdots ,z_{q+r})=\underset{\sigma (
q+1) <\cdots <\sigma ( q+r) }{\underset{\sigma ( 1) <\cdots <\sigma ( q) }{%
\sum }}sgn( \sigma )\rho ( z) ( \omega ( z_{\sigma ( 1) },\cdots ,z_{\sigma
( q) }) ) \\
&\qquad \cdot \theta ( z_{\sigma ( q+1) },\cdots ,z_{\sigma ( q+r) }) +%
\underset{\sigma ( q+1) <\cdots <\sigma ( q+r) }{\underset{\sigma ( 1)
<\cdots <\sigma ( q) }{\sum }}sgn( \sigma )\omega ( z_{\sigma ( 1) },\cdots
,z_{\sigma ( q) }) \\
&\qquad \cdot \rho ( z) ( \theta ( z_{\sigma ( q+1) },\cdots ,z_{\sigma (
q+r) }) ) -\underset{\sigma ( q+1) <\cdots <\sigma ( q+r) }{\underset{\sigma
( 1) <\cdots <\sigma ( q) }{\sum }}sgn( \sigma ) \\
&\qquad \overset{q}{\underset{i=1}{\sum }}\omega ( z_{\sigma ( 1) },\cdots
,[ z,z_{\sigma ( i) }] _{A},\cdots ,z_{\sigma ( q) }) \cdot \theta (
z_{\sigma ( q+1) },\cdots ,z_{\sigma ( q+r) }) \\
&-\underset{\sigma ( q+1) <\cdots <\sigma ( q+r) }{\underset{\sigma ( 1)
<\cdots <\sigma ( q) }{\sum }}sgn( \sigma ) \overset{q+r}{\underset{i=q+1}{%
\sum }}\omega ( z_{\sigma ( 1) },\cdots ,z_{\sigma ( q) }) \\
&\qquad \cdot \theta ( z_{\sigma ( q+1) },\cdots ,[ z,z_{\sigma ( i) }]
_{A},\cdots ,z_{\sigma ( q+r) }),
\end{align*}
and consequently 
\begin{align*}
&L_{z}( \omega \wedge \theta ) ( z_{1},\cdots ,z_{q+r})=\underset{\sigma ( q+1) <\cdots <\sigma ( q+r) }{\underset{\sigma ( 1)
<\cdots <\sigma ( q) }{\sum }}sgn( \sigma ) L_{z}\omega ( z_{\sigma ( 1)
},\cdots ,z_{\sigma ( q) }) \\
&\qquad \cdot \theta ( z_{\sigma ( q+1) },\cdots ,z_{\sigma ( q+r) }) +%
\underset{\sigma ( q+1) <\cdots <\sigma ( q+r) }{\underset{\sigma ( 1)
<\cdots <\sigma ( q) }{\sum }}sgn( \sigma ) \overset{q+r}{\underset{i=q+1}{%
\sum }}\omega ( z_{\sigma ( 1) },\cdots ,z_{\sigma ( q) }) \\
&\qquad \cdot L_{z}\theta ( z_{\sigma ( q+1) },\cdots ,z_{\sigma ( q+r) }) \\
&=( L_{z}\omega \wedge \theta +\omega \wedge L_{z}\theta ) ( z_{1},\cdots
,z_{q+r}).
\end{align*}
\end{proof}
\begin{definition}
For any $z\in A$, the operator 
\begin{equation*}
\begin{array}{rcl}
\Lambda ( A) & ^{\underrightarrow{\ \ i_{z}\ \ }} & \Lambda ( A) \\ 
\Lambda ^{q}( A) \ni \omega & \longmapsto & i_{z}\omega \in \Lambda ^{q-1}(
A) ,%
\end{array}%
\end{equation*}%
given by 
\begin{equation}
\begin{array}{c}
i_{z}\omega ( z_{2},\cdots ,z_{q}) =\omega ( z,z_{2},\cdots ,z_{q}) ,%
\end{array}%
\end{equation}%
for any $z_{2},\cdots ,z_{q}\in A$, is called the interior product
associated to the element~$z$. Moreover, for any $f\in \mathcal{F}$, we
define $\ i_{z}f=0.$
\end{definition}
\begin{theorem}
If $z\in A$, then for any $\omega \in \Lambda ^{q}(
A) $ and $\theta \in \Lambda ^{r}( A) $ we obtain 
\begin{equation}
\begin{array}{c}
i_{z}( \omega \wedge \theta ) =i_{z}\omega \wedge \theta +( -1) ^{q}\omega
\wedge i_{z}\theta .%
\end{array}%
\end{equation}
\end{theorem}
\begin{proof}
Let $z_{1},\cdots ,z_{q+r}\in A$. Then using the above definition we get 
\begin{equation*}
\begin{array}{l}
i_{z_{1}}( \omega \wedge \theta ) ( z_{2},\cdots ,z_{q+r}) =( \omega \wedge
\theta ) ( z_{1},z_{2},\cdots ,z_{q+r}) \\ 
=\underset{\sigma ( q+1) <\cdots <\sigma ( q+r) }{\underset{\sigma ( 1)
<\cdots <\sigma ( q) }{\sum }}sgn( \sigma )\omega ( z_{\sigma ( 1) },\cdots
,z_{\sigma ( q) }) \cdot \theta ( z_{\sigma ( q+1) },\cdots ,z_{\sigma (
q+r) }) \\ 
=\underset{\sigma ( q+1) <\cdots <\sigma ( q+r) }{\underset{1=\sigma ( 1)
<\sigma ( 2) <\cdots <\sigma ( q) }{\sum }}sgn( \sigma )\omega (
z_{1},z_{\sigma ( 2) },\cdots ,z_{\sigma ( q) }) \cdot \theta ( z_{\sigma (
q+1) },\cdots ,z_{\sigma ( q+r) }) \\ 
+\!\!\!\!\!\!\!\!\underset{1=\sigma ( q+1) <\sigma ( q+2) <\cdots <\sigma (
q+r) }{\underset{\sigma ( 1) <\cdots <\sigma ( q) }{\sum }}%
\!\!\!\!\!\!\!sgn( \sigma )\omega ( z_{\sigma ( 1) },\cdots ,z_{\sigma ( q)
}) \cdot \theta ( z_{1},z_{\sigma ( q+2) },\cdots ,z_{\sigma ( q+r) }) \\ 
=\underset{\sigma ( q+1) <\cdots <\sigma ( q+r) }{\underset{\sigma ( 2)
<\cdots <\sigma ( q) }{\sum }}sgn( \sigma )i_{z_{1}}\omega ( z_{\sigma ( 2)
},\cdots ,z_{\sigma ( q) }) \cdot \theta ( z_{\sigma ( q+1) },\cdots
,z_{\sigma ( q+r) }) \\ 
+\underset{\sigma ( q+2) <\cdots <\sigma ( q+r) }{\underset{\sigma ( 1)
<\cdots <\sigma ( q) }{\sum }}sgn( \sigma )\omega ( z_{\sigma ( 1) },\cdots
,z_{\sigma ( q) }) \cdot i_{z_{1}}\theta ( z_{\sigma ( q+2) },\cdots
,z_{\sigma ( q+r) }).
\end{array}%
\end{equation*}
In the second sum, we have the permutation%
\begin{equation*}
\sigma =( 
\begin{array}{ccccccc}
1 & \cdots & q & q+1 & q+2 & \cdots & q+r \\ 
\sigma ( 1) & \cdots & \sigma ( q) & 1 & \sigma ( q+2) & \cdots & \sigma (
q+r)%
\end{array}%
) .
\end{equation*}
We observe that $\sigma =\tau \circ \tau ^{\prime }$, where%
\begin{equation*}
\tau =( 
\begin{array}{ccccccc}
1 & 2 & \cdots & q+1 & q+2 & \cdots & q+r \\ 
1 & \sigma ( 1) & \cdots & \sigma ( q) & \sigma ( q+2) & \cdots & \sigma (
q+r)%
\end{array}%
),
\end{equation*}%
and 
\begin{equation*}
\tau ^{\prime }=( 
\begin{array}{cccccccc}
1 & 2 & \cdots & q & q+1 & q+2 & \cdots & q+r \\ 
2 & 3 & \cdots & q+1 & 1 & q+2 & \cdots & q+r%
\end{array}%
) .
\end{equation*}
Since $\tau ( 2) <\cdots <\tau ( q+1) $ and $\tau ^{\prime }$ has $q$
inversions, it results that 
\begin{equation*}
sgn( \sigma ) =( -1)^{q}sgn( \tau ) .
\end{equation*}
Therefore 
\begin{equation*}
\begin{array}{l}
i_{z_{1}}( \omega \wedge \theta ) ( z_{2},\cdots ,z_{q+r}) =(
i_{z_{1}}\omega \wedge \theta ) ( z_{2},\cdots ,z_{q+r}) \\ 
+( -1) ^{q}\!\!\!\!\!\!\!\!\underset{\tau ( q+2) <\cdots <\tau ( q+r) }{%
\underset{\tau ( 2) <\cdots <\tau ( q) }{\sum }}\!\!\!\!\!\!\!\!sgn( \tau )
\cdot \omega ( z_{\tau ( 2) },\cdots ,z_{\tau ( q) }) \cdot i_{z_{1}}\theta
( z_{\tau ( q+2) },\cdots ,z_{\tau ( q+r) }) \\ 
=( i_{z_{1}}\omega \wedge \theta ) ( z_{2},\cdots ,z_{q+r}) +( -1) ^{q}(
\omega \wedge i_{z_{1}}\theta ) ( z_{2},\cdots ,z_{q+r}),
\end{array}%
\end{equation*}
which completes the proof.
\end{proof}
\begin{theorem}
For any $z,v\in A$, we have 
\begin{equation}
\begin{array}{c}
L_{v}\circ i_{z}-i_{z}\circ L_{v}=i_{[v, z]_{A}}.%
\end{array}%
\end{equation}
\end{theorem}

\begin{proof}
Let $\omega \in \Lambda ^{q}( A) $. Then 
\begin{equation*}
\begin{array}{l}
i_{z}( L_{v}\omega ) ( z_{2},\cdots z_{q}) =L_{v}\omega ( z,z_{2},\cdots
z_{q}) \\ 
=\rho ( v) ( \omega ( z,z_{2},\cdots ,z_{q}) ) -\omega ( [ v,z]
_{A},z_{2},\cdots ,z_{q}) \\ 
-\overset{q}{\underset{i=2}{\sum }}\omega ( ( z,z_{2},\cdots ,[ v,z_{i}]
_{A},\cdots ,z_{q}) ) \\ 
=\rho ( v) ( i_{z}\omega ( z_{2},\cdots ,z_{q}) ) -\overset{q}{\underset{i=2}%
{\sum }}i_{z}\omega ( z_{2},\cdots ,[ v,z_{i}] _{A},\cdots ,z_{q}) \\ 
-i_{[ v,z] _{A}}( z_{2},\cdots ,z_{q}) =( L_{v}( i_{z}\omega ) -i_{[ v,z]
_{A}}) ( z_{2},\cdots ,z_{q}),%
\end{array}%
\end{equation*}%
for any $z_{2},\cdots ,z_{q}\in A$.
\end{proof}

\begin{definition}
If $f\in \mathcal{F}$ and $z\in A,$ then the exterior differential operator
is defined by 
\begin{equation}
\begin{array}{c}
d^{A}f( z) =\rho ( z) f.%
\end{array}%
\end{equation}
\end{definition}

\begin{theorem}
\label{Th34} The operator 
\begin{equation*}
\begin{array}{c}
\begin{array}{ccc}
\Lambda ^{q}\mathbf{\ }( A) & ^{\underrightarrow{\,\ d^{A}\,\ }} & \Lambda
^{q+1}\mathbf{\ }( A) \\ 
\omega & \longmapsto & d^{A}\omega%
\end{array}%
\end{array}
,
\end{equation*}%
defined by 
\begin{equation}\label{29}
\begin{array}{l}
d^{A}\omega ( z_{0},z_{1},\cdots ,z_{q}) =\overset{q}{\underset{i=0}{\sum }}%
( -1) ^{i}\rho ( z_{i}) ( \omega ( z_{0},z_{1},\cdots ,\hat{z}_{i},\cdots
,z_{q}) ) \\ 
~\ \ \ \ \ \ \ \ \ \ \ \ \ \ \ \ \ \ \ \ \ \ \ \ \ \ +\underset{i<j}{\sum }(
-1) ^{i+j}\omega ( [ z_{i},z_{j}] _{A},z_{0},z_{1},\cdots ,\hat{z}%
_{i},\cdots ,\hat{z}_{j},\cdots ,z_{q}) ,%
\end{array}%
\end{equation}%
for any $z_{0},z_{1},\cdots ,z_{q}\in A,$ is unique with the following
property:%
\begin{equation}  \label{form3}
\begin{array}{c}
L_{z}=d^{A}\circ i_{z}+i_{z}\circ d^{A},\ \ \ \ \forall z\in A.%
\end{array}%
\end{equation}
\end{theorem}

\begin{proof}
At first we prove that the equation (\ref{form3}) holds. We get 
\begin{align*}
&( i_{z_{0}}\circ d^{A}) \omega ( z_{1},\cdots ,z_{q})=d^{A}\omega (
z_{0},z_{1},\cdots ,z_{q}) \\ 
&=\overset{q}{\underset{i=0}{\sum }}( -1) ^{i}\rho ( z_{i}) ( \omega (
z_{0},z_{1},\cdots ,\hat{z}_{i},\cdots ,z_{q}) ) \\ 
&+\underset{0\leq i<j}{\sum }( -1) ^{i+j}\omega ( [ z_{i},z_{j}]
_{A},z_{0},z_{1},\cdots ,\hat{z}_{i},\cdots ,\hat{z}_{j},\cdots ,z_{q}) \\ 
&=\rho ( z_{0}) ( \omega ( z_{1},\cdots ,z_{q}) ) \\ 
&+\overset{q}{\underset{i=1}{\sum }}( -1) ^{i}\rho ( z_{i}) ( \omega (
z_{0},z_{1},\cdots ,\hat{z}_{i},\cdots ,z_{q}) ) \\ 
&+\underset{i=1}{\overset{q}{\sum }}( -1) ^{i}\omega ( [ z_{0},z_{i}]
_{A},z_{1},\cdots ,\hat{z}_{i},\cdots ,z_{q}) \\ 
&+\underset{1\leq i<j}{\sum }( -1) ^{i+j}\omega ( [ z_{i},z_{j}]
_{A},z_{0},z_{1},\cdots ,\hat{z}_{i},\cdots ,\hat{z}_{j},\cdots ,z_{q}),
\end{align*}
which gives us
\begin{align*} 
&( i_{z_{0}}\circ d^{A}) \omega ( z_{1},\cdots ,z_{q})=\rho ( z_{0}) ( \omega ( z_{1},\cdots ,z_{q}) )\\ 
&-\underset{i=1}{\overset{q}{\sum }}\omega ( z_{1},\cdots ,[ z_{0},z_{i}]
_{A},\cdots ,z_{q}) \\ 
&-\overset{q}{\underset{i=1}{\sum }}( -1) ^{i-1}\rho ( z_{i}) (
i_{z_{0}}\omega ( ( z_{1},\cdots ,\hat{z}_{i},\cdots ,z_{q}) ) ) \\ 
&-\underset{1\leq i<j}{\sum }( -1) ^{i+j-2}i_{z_{0}}\omega ( ( [ z_{i},z_{j}]
_{A},z_{1},\cdots ,\hat{z}_{i},\cdots ,\hat{z}_{j},\cdots ,z_{q}) ) \\ 
&=( L_{z_{0}}-d^{A}\circ i_{z_{0}}) \omega ( z_{1},\cdots ,z_{q}) ,%
\end{align*}%
for any $z_{0},z_{1},\cdots ,z_{q}\in A$. Thus (\ref{form3}) holds. Now, we
verify the uniqueness of the operator $d^{A}$. Let $d^{\prime A}$ be an
another exterior differentiation operator satisfying the property (\ref%
{form3}). We consider the set 
\begin{equation*}
S=\left\{ q\in \mathbb{N}|d^{A}\omega =d^{\prime A}\omega ,~\forall \omega
\in \Lambda ^{q}( ( F,\nu ,N) ) A\right\},
\end{equation*}
and we let $z\in ( A,[ ,] _{A}~,\rho ) $. We observe that (\ref%
{form3}) is equivalent with 
\begin{equation}  \label{0}
\begin{array}{c}
i_{z}\circ ( d^{A}-d^{\prime A}) +( d^{A}-d^{\prime A}) \circ i_{z}=0.%
\end{array}%
\end{equation}
Since $i_{z}f=0,$ for any $f\in \mathcal{F},$ it results that 
\begin{equation*}
\begin{array}{c}
( d^{A}-d^{\prime A}) ( f) ( z) =0,~\forall f\in \mathcal{F}.%
\end{array}%
\end{equation*}
Therefore, we obtain $0\in S$. We now prove that if $q\in S$ then $q+1\in S$. 
Let $\omega \in \Lambda ^{p+1}( A) $. Since $i_{z}\omega \in \Lambda ^{q}(
A) $, using the equality (\ref{0}), it results that 
\begin{equation*}
i_{z}\circ ( d^{A}-d^{\prime A}) \omega =0,
\end{equation*}
which implies $( ( d^{A}-d^{\prime A}) \omega ) ( z_{0},z_{1},\cdots
,z_{q}) =0,$ for any $z_{1},\cdots ,z_{q}\in A.$ Therefore $d^{A}\omega
=d^{\prime A}\omega .$ Namely $q+1\in S$. Thus the Peano's Axiom implies that $S=\mathbb{N}$.
Therefore, the uniqueness is verified.
\end{proof}
The operator given by Theorem \ref{Th34} will be called\emph{\ the exterior
differentiation ope\-ra\-tor for the exterior differential algebra of the
generalized Lie }$\mathcal{F}$-\emph{algebra }$( A,[ ,] _{A}~,\rho
)$. Using (\ref{29}) we deduce that if $\omega =\omega _{\alpha _{1}\cdots \alpha _{q}}t^{\alpha _{1}}\wedge
\cdots \wedge t^{\alpha _{q}}\in \Lambda ^{q}( A) $, then 
\begin{align*}
d^{A}\omega ( t_{\alpha _{0}},t_{\alpha _{1}},\cdots ,t_{\alpha _{q}})&=
\overset{q}{\underset{i=0}{\sum }}( -1) ^{i}\rho _{\alpha _{i}}^{k}\partial
_{k}( \omega _{\alpha _{0},\cdots ,\widehat{\alpha _{i}}\cdots \alpha _{q}})
\\
&\ \ \ +\underset{i<j}{\sum }( -1) ^{i+j}L_{\alpha _{i}\alpha _{j}}^{\alpha
}\cdot \omega _{\alpha ,\alpha _{0},\cdots ,\widehat{\alpha _{i}},\cdots ,%
\widehat{\alpha _{j}},\cdots ,\alpha _{q}}.
\end{align*}
Therefore, $d^A\omega$ has the following locally expression:
\begin{align*}
&d^{A}\omega=\Big(\overset{q}{\underset{i=0}{\sum }}( -1) ^{i}\rho _{\alpha
_{i}}^{k}\partial _{k}( \omega _{\alpha _{0},\cdots ,\widehat{\alpha _{i}}%
\cdots \alpha _{q}})  \notag \\
&\ \ \ \ \ \ \ \ \ +\underset{i<j}{\sum }( -1) ^{i+j}L_{\alpha _{i}\alpha
_{j}}^{\alpha }\cdot \omega _{\alpha ,\alpha _{0},\cdots ,\widehat{\alpha
_{i}},\cdots ,\widehat{\alpha _{j}},\cdots ,\alpha _{q}}\Big) t^{\alpha
_{0}}\wedge t^{\alpha _{1}}\wedge \cdots \wedge t^{\alpha _{q}}.
\end{align*}
\begin{theorem}
\label{38} The exterior differentiation operator $d^{A}$  has the following properties: 
\begin{equation*}
\begin{array}{c}
(i)\ \ d^{A}( \omega \wedge \theta ) =d^{A}\omega \wedge \theta +( -1)
^{q}\omega \wedge d^{A}\theta,\ \ \ \forall\omega \in\Lambda ^{q}( A), \
\forall \theta \in \Lambda ^{r}( A),%
\end{array}%
\end{equation*}
\begin{equation*}
\begin{array}{c}
\hspace{-5.6cm}(ii)\ \ L_{z}\circ d^{A}=d^{A}\circ L_{z}, \ \ \forall z\in A,%
\end{array}%
\end{equation*}
\begin{equation*}
\hspace{-8.2cm}(iii)\ \ d^{A}\circ d^{A}=0.
\end{equation*}
\end{theorem}

\begin{proof}
(i)\ Let 
\begin{equation*}
S=\left\{ q\in \mathbb{N}|d^{A}( \omega \wedge \theta ) =d^{A}\omega \wedge
\theta +( -1) ^{q}\omega \wedge d^{A}\theta ,~\forall \omega \in \Lambda
^{q}( A) \right\}.
\end{equation*}
Since 
\begin{equation*}
\begin{array}{l}
d^{A}( f\wedge \omega) ( z,v) =d^{A}( f\cdot \omega) ( z,v) \\ 
=\rho ( z) ( f\omega ( v) ) -\rho ( v) ( f\omega ( z) ) -f\omega ( [ z,v]
_{A}) \\ 
=\rho ( z) ( f) \cdot \omega ( v) +f\cdot \rho ( z) ( \omega ( v) ) \\ 
-\rho ( v) ( f) \cdot \omega ( z) -f\cdot \rho ( v) ( \omega ( z) ) -f\omega
( [ z,v] _{A}) \\ 
=d^{A}f( z) \cdot \omega ( v) -d^{A}f( v) \cdot \omega ( z) +f\cdot
d^{A}\omega ( z,v) \\ 
=( d^{A}f\wedge \omega ) ( z,v) +( -1) ^{0}f\cdot d^{A}\omega ( z,v) \\ 
=( d^{A}f\wedge \omega ) ( z,v) +( -1) ^{0}( f\wedge d^{A}\omega ) ( z,v)
,~\forall z,v\in A,%
\end{array}%
\end{equation*}%
then we deduce $0\in S$. We now prove that if $q\in S$, then $q+1\in S$. Without loss of generality, we consider that $\theta \in \Lambda
^{r}( A)$. Then we get
\begin{equation*}
\begin{array}{l}
d^{A}( \omega \wedge \theta ) ( z_{0},z_{1},\cdots ,z_{q+r}) =i_{z_{0}}\circ
d^{A}( \omega \wedge \theta ) ( z_{1},\cdots ,z_{q+r}) \\ 
=L_{z_{0}}( \omega \wedge \theta ) ( z_{1},\cdots ,z_{q+r}) -d^{A}\circ
i_{z_{0}}( \omega \wedge \theta ) ( z_{1},\cdots ,z_{q+r}) \\ 
=( L_{z_{0}}\omega \wedge \theta +\omega \wedge L_{z_{0}}\theta ) (
z_{1},\cdots ,z_{q+r}) \\ 
-[ d^{A}\circ ( i_{z_{0}}\omega \wedge \theta +( -1) ^{q}\omega \wedge
i_{z_{0}}\theta ) ] ( z_{1},\cdots ,z_{q+r}) \\ 
=( L_{z_{0}}\omega \wedge \theta +\omega \wedge L_{z_{0}}\theta -(
d^{A}\circ i_{z_{0}}\omega ) \wedge \theta ) ( z_{1},\cdots ,z_{q+r}) \\ 
-( ( -1) ^{q-1}i_{z_{0}}\omega \wedge d^{A}\theta +( -1) ^{q}d^{A}\omega
\wedge i_{z_{0}}\theta ) ( z_{1},\cdots ,z_{q+r}) \\ 
-( -1) ^{2q}\omega \wedge d^{A}\circ i_{z_{0}}\theta ( z_{1},\cdots ,z_{q+r})
\\ 
=( ( L_{z_{0}}\omega -d^{A}\circ i_{z_{0}}\omega ) \wedge \theta ) (
z_{1},\cdots ,z_{q+r}) \\ 
+\omega \wedge ( L_{z_{0}}\theta -d^{A}\circ i_{z_{0}}\theta ) (
z_{1},\cdots ,z_{q+r}) \\ 
+( ( -1) ^{q}i_{z_{0}}\omega \wedge d^{A}\theta -( -1) ^{q}d^{A}\omega
\wedge i_{z_{0}}\theta ) ( z_{1},\cdots ,z_{q+r}) \\ 
=[ ( ( i_{z_{0}}\circ d^{A}) \omega ) \wedge \theta +( -1) ^{q+1}d^{A}\omega
\wedge i_{z_{0}}\theta ] ( z_{1},\cdots ,z_{q+r}) \\ 
+[ \omega \wedge ( ( i_{z_{0}}\circ d^{A}) \theta ) +( -1)
^{q}i_{z_{0}}\omega \wedge d^{A}\theta ] ( z_{1},\cdots ,z_{q+r}) \\ 
=[ i_{z_{0}}( d^{A}\omega \wedge \theta ) +( -1) ^{q}i_{z_{0}}( \omega
\wedge d^{A}\theta ) ] ( z_{1},\cdots ,z_{q+r}) \\ 
=[ d^{A}\omega \wedge \theta +( -1) ^{q}\omega \wedge d^{A}\theta ] (
z_{1},\cdots ,z_{q+r}) ,%
\end{array}%
\end{equation*}%
for any $z_{0},z_{1},\cdots ,z_{q+r}\in A$, which gives us $q+1\in S$. Thus using
the Peano's Axiom we deduce $S=\mathbb{N}$.\\
(ii)\ Let $z\in A$ and we consider 
\begin{equation*}
S=\left\{ q\in \mathbb{N}|( L_{z}\circ d^{A}) \omega =( d^{A}\circ L_{z})
\omega ,~\forall \omega \in \Lambda ^{q}( A) \right\}.
\end{equation*}
If $f\in \mathcal{F}$, then we get 
\begin{equation*}
\begin{array}{l}
( d^{A}\circ L_{z}) f( v) =i_{v}\circ ( d^{A}\circ L_{z}) f=( i_{v}\circ
d^{A}) \circ L_{z}f \\ 
=( L_{v}\circ L_{z}) f-( ( d^{A}\circ i_{v}) \circ L_{z}) f \\ 
=( L_{v}\circ L_{z}) f-L_{[ z,v] _{A}}f+d^{A}\circ i_{[ z,v]
_{A}}f-d^{A}\circ L_{z}( i_{v}f) \\ 
=( L_{v}\circ L_{z}) f-L_{[ z,v] _{A}}f+d^{A}\circ i_{[ z,v] _{A}}f-0 \\ 
=( L_{v}\circ L_{z}) f-L_{[ z,v] _{A}}f+d^{A}\circ i_{[ z,v]
_{A}}f-L_{z}\circ d^{A}( i_{v}f) \\ 
=( L_{z}\circ i_{v}) ( d^{A}f) -L_{[ z,v] _{A}}f+d^{A}\circ i_{[ z,v] _{A}}f
\\ 
=( i_{v}\circ L_{z}) ( d^{A}f) +L_{[ z,v] _{A}}f-L_{[ z,v] _{A}}f \\ 
=i_{v}\circ ( L_{z}\circ d^{A}) f=( L_{z}\circ d^{A}) f( v) ,~\forall v\in A,%
\end{array}%
\end{equation*}%
which results that $0\in S$. We now show that if $q\in S$, then $q+1\in S$. Let $\omega \in \Lambda ^{q}( A)$. Then
\begin{align*}
&( d^{A}\circ L_{z}) \omega ( z_{0},z_{1},\cdots ,z_{q})=i_{z_{0}}\circ (
d^{A}\circ L_{z}) \omega ( z_{1},\cdots ,z_{q})\\ 
&=( i_{z_{0}}\circ d^{A}) \circ L_{z}\omega ( z_{1},\cdots ,z_{q})\\ 
&=[ ( L_{z_{0}}\circ L_{z}) \omega -( ( d^{A}\circ i_{z_{0}}) \circ L_{z})
\omega ] ( z_{1},\cdots ,z_{q})\\ 
&=[ ( L_{z_{0}}\circ L_{z}) \omega -L_{[ z,z_{0}] _{A}}\omega ] (
z_{1},\cdots ,z_{q})\\ 
&+[ d^{A}\circ i_{[ z,z_{0}] _{A}}\omega -d^{A}\circ L_{z}( i_{z_{0}}\omega )
] ( z_{1},\cdots ,z_{q}).
\end{align*}
Using (ii) in the above equation we obtain
\begin{align*}
&( d^{A}\circ L_{z}) \omega ( z_{0},z_{1},\cdots ,z_{q})=[ ( L_{z_{0}}\circ L_{z}) \omega -L_{[ z,z_{0}] _{A}}\omega
] ( z_{1},\cdots ,z_{q})\\ 
&+[ d^{A}\circ i_{[ z,z_{0}] _{A}}\omega -L_{z}\circ d^{A}( i_{z_{0}}\omega )
] ( z_{1},\cdots ,z_{q})\\ 
&=[ ( L_{z}\circ i_{z_{0}}) ( d^{A}\omega ) -L_{[ z,z_{0}] _{A}}\omega
+d^{F}\circ i_{[ z,z_{0}] _{A}}\omega ] ( z_{1},\cdots ,z_{q})\\ 
&=[ ( i_{z_{0}}\circ L_{z}) ( d^{A}\omega ) +L_{[ z,z_{0}] _{A}}\omega -L_{[
z,z_{0}] _{A}}\omega ] ( z_{1},\cdots ,z_{q}) \\ 
&=i_{z_{0}}\circ ( L_{z}\circ d^{A}) \omega ( z_{1},\cdots ,z_{q})\\ 
&=( L_{z}\circ d^{A}) \omega ( z_{0},z_{1},\cdots ,z_{q}) ,~\forall
z_{0},z_{1},\cdots ,z_{q}\in A,%
\end{align*}%
which implies $q+1\in S$. Using the Peano's Axiom we result that $S=\mathbb{N}$.%
\newline
(iii)\ It is remarked that 
\begin{equation*}
\begin{array}{l}
i_{z}\circ ( d^{A}\circ d^{A}) =( i_{z}\circ d^{A}) \circ d^{A}=L_{z}\circ
d^{A}-( d^{A}\circ i_{z}) \circ d^{A}\vspace*{1mm} \\ 
=L_{z}\circ d^{A}-d^{A}\circ L_{z}+d^{A}\circ ( d^{A}\circ i_{z}) =(
d^{A}\circ d^{A}) \circ i_{z},%
\end{array}%
\end{equation*}%
for any $z\in A$. Now, let $\omega \in \Lambda ^{q}( A)$. Then we get 
\begin{equation*}
\begin{array}{l}
( d^{A}\circ d^{A}) \omega ( z_{1},\cdots ,z_{q+2}) =i_{z_{q+2}}\circ \cdots
\circ i_{z_{1}}\circ ( d^{A}\circ d^{A}) \omega\\ 
=i_{z_{q+2}}\circ ( d^{A}\circ d^{A}) \circ i_{z_{q+1}}( \omega (
z_{1},\cdots ,z_{q}) ) \vspace*{1mm} \\ 
=i_{z_{q+2}}\circ ( d^{A}\circ d^{A}) ( 0) =0,~\forall z_{1},\cdots
,z_{q+2}\in A.%
\end{array}%
\end{equation*}
\end{proof}

\begin{theorem}
If $d^{A}$ is the exterior differentiation operator for the exterior
differential $\mathcal{F}$-algebra $(\Lambda (A),\wedge )$, then we
obtain the structure equations of Maurer-Cartan type 
\begin{equation}  \label{36}
\begin{array}{c}
d^{A}t^{\alpha }=-\displaystyle\frac{1}{2}L_{\beta \gamma }^{\alpha
}t^{\beta }\wedge t^{\gamma },%
\end{array}%
\end{equation}%
where $\{ t^{\alpha }\} ~$ is the coframe of the Lie algebra $( A,[
,] _{A})$.
\end{theorem}

\begin{proof}
Without restriction of generality, we admit that the set of indices of the
base of $A$ is ordered. Let $\alpha $ be arbitrary. Since 
\begin{equation*}
\begin{array}{c}
d^{A}t^{\alpha }( t_{\beta },t_{\gamma }) =-L_{\beta \gamma }^{\alpha
},~\forall \beta ,\gamma,%
\end{array}%
\end{equation*}%
then 
\begin{equation}  \label{F1}
\begin{array}{c}
d^{A}t^{\alpha }=-\underset{\beta <\gamma }{\sum }L_{\beta \gamma }^{\alpha
}t^{\beta }\wedge t^{\gamma }.%
\end{array}%
\end{equation}
Since $L_{\beta \gamma }^{\alpha }=-L_{\gamma \beta }^{\alpha }$ and $%
t^{\beta }\wedge t^{\gamma }=-t^{\gamma }\wedge t^{\beta }$, it results that 
\begin{equation}  \label{F2}
\begin{array}{c}
\underset{\beta <\gamma }{\sum }L_{\beta \gamma }^{\alpha }t^{\beta }\wedge
t^{\gamma }=\displaystyle\frac{1}{2}L_{\beta \gamma }^{\alpha }t^{\beta
}\wedge t^{\gamma }.%
\end{array}%
\end{equation}
(\ref{F1}) and (\ref{F2}) imply (\ref{36}).
\end{proof}
Equation (\ref{36}) will be called \emph{the structure equations of
Maurer-Cartan type associa\-ted to the generalized Lie }$\mathcal{F}$-algebra%
\emph{\ }$( A,[ ,] _{A},\rho )$.
\begin{corollary}
\emph{If }$d^{F}$ \emph{is the exterior differentiation operator for the
exterior differential} $\mathcal{F}(N)$\emph{-algebra} $(\Lambda (F,\nu
,N),\wedge )$, \emph{then locally we obtain the structure equations
of Maurer-Cartan type }%
\begin{equation*}
\begin{array}{c}
d^{F}t^{\alpha }=-\displaystyle\frac{1}{2}L_{\beta \gamma }^{\alpha
}t^{\beta }\wedge t^{\gamma },~\alpha \in \overline{1,p},
\end{array}%
\leqno(\mathcal{C}_{1})
\end{equation*}%
\emph{and\ }%
\begin{equation*}
\begin{array}{c}
d^{F}\varkappa ^{\tilde{\imath}}=\theta _{\alpha }^{\tilde{\imath}}t^{\alpha
},~\tilde{\imath}\in \overline{1,n},%
\end{array}%
\leqno(\mathcal{C}_{2})
\end{equation*}%
\emph{where }$\left\{ t^{\alpha },\alpha \in \overline{1,p}\right\} ~$\emph{%
is the coframe of the vector bundle }$( F,\nu ,N) .$\bigskip
\end{corollary}
This equations will be called \emph{the structure equations of Maurer-Cartan
type associa\-ted to the generalized Lie algebroid }$( ( F,\nu ,N) ,[ ,]
_{F,h},( \rho ,\eta ) )$. In the particular case of Lie algebroids, $( \eta ,h) =( Id_{M},Id_{M}) ,$
the structure equations of Maurer-Cartan type become%
\begin{equation*}
\begin{array}{c}
d^{F}t^{\alpha }=-\displaystyle\frac{1}{2}L_{\beta \gamma }^{\alpha
}t^{\beta }\wedge t^{\gamma },~\alpha \in \overline{1,p},
\end{array}%
\leqno(\mathcal{C}_{1}^{\prime })
\end{equation*}%
\emph{and\ }%
\begin{equation*}
\begin{array}{c}
d^{F}x^{i}=\rho _{\alpha }^{i}t^{\alpha },~i\in \overline{1,m}.%
\end{array}%
\leqno(\mathcal{C}_{2}^{\prime })
\end{equation*}
Also, in the particular case of standard Lie algebroid, $\rho =Id_{TM},$ the
structure equations of Maurer-Cartan type become%
\begin{equation*}
\begin{array}{c}
d^{TM}dx^{i}=0,~i\in \overline{1,m},
\end{array}%
\leqno(\mathcal{C}_{1}^{\prime \prime })
\end{equation*}%
\emph{and\ }%
\begin{equation*}
\begin{array}{c}
d^{TM}x^{i}=dx^{i},~i\in \overline{1,m}.%
\end{array}%
\leqno(\mathcal{C}_{2}^{\prime \prime })
\end{equation*}

\begin{definition}
For any generalized Lie $\mathcal{F}$-algebras morphism $\varphi $ from $(
A,[ ,] _{A},\rho ) $ to $( A^{\prime },[ ,] _{A^{\prime
}},\rho ^{\prime }) $ we define the application 
\begin{equation*}
\begin{array}{ccc}
\Lambda ^{q}( A^{\prime }) & ^{\underrightarrow{\ \varphi ^{\ast }\ }} & 
\Lambda ^{q}( A) \\ 
\omega ^{\prime } & \longmapsto & \varphi ^{\ast }\omega ^{\prime }%
\end{array}%
,
\end{equation*}%
where 
\begin{equation*}
\begin{array}{c}
( \varphi ^{\ast }\omega ^{\prime }) ( z_{1},\cdots ,z_{q}) =\omega ^{\prime
}( \varphi ( z_{1}) ,\cdots ,\varphi ( z_{q}) ) ,%
\end{array}%
\end{equation*}%
for any $z_{1},\cdots ,z_{q}\in A.$
\end{definition}

\begin{theorem}
If $\varphi $ is a generalized Lie $\mathcal{F}$-algebras morphism from $(
A,[ ,] _{A},\rho ) $ to $( A^{\prime },[ ,] _{A^{\prime
}},\rho ^{\prime }) $, then 
\begin{equation*}
(i)\ \varphi ^{\ast }( \omega ^{\prime }\wedge \theta ^{\prime }) =\varphi
^{\ast }\omega ^{\prime }\wedge \varphi ^{\ast }\theta ^{\prime },\ \ (ii)\
i_{z}( \varphi ^{\ast }\omega ^{\prime }) =\varphi ^{\ast }( i_{\varphi ( z)
}\omega ^{\prime }),\ \ (iii)\ \varphi ^{\ast }\circ d^{A^{\prime
}}=d^{A}\circ \varphi ^{\ast },
\end{equation*}
where $z\in A$, $\omega ^{\prime }\in \Lambda ^{q}( A^{\prime }) $ and $%
\theta ^{\prime }\in \Lambda ^{r}( A^{\prime })$.
\end{theorem}

\begin{proof}
Let $\omega ^{\prime }\in \Lambda ^{q}( A^{\prime }) $ and $\theta ^{\prime
}\in \Lambda ^{r}( A^{\prime }) $. Then we get 
\begin{equation*}
\begin{array}{l}
\displaystyle\varphi ^{\ast }( \omega ^{\prime }\wedge \theta ^{\prime }) (
z_{1},\cdots ,z_{q+r}) =( \omega ^{\prime }\wedge \theta ^{\prime }) (
\varphi ( z_{1}) ,\cdots ,\varphi ( z_{q+r}) ) \vspace*{1mm} \\ 
\qquad \displaystyle=\frac{1}{( q+r) !}\underset{\sigma \in \Sigma _{q+r}}{%
\sum }sgn( \sigma ) \cdot \omega ^{\prime }( \varphi ( z_{1}) ,\cdots
,\varphi ( z_{q}) ) \vspace*{1mm}\cdot \theta ^{\prime }( \varphi ( z_{q+1})
,\cdots ,\varphi ( z_{q+r}) ) \vspace*{1mm} \\ 
\qquad \displaystyle=\frac{1}{( q+r) !}\underset{\sigma \in \Sigma _{q+r}}{%
\sum }sgn( \sigma ) \cdot \varphi ^{\ast }\omega ^{\prime }( z_{1},\cdots
,z_{q}) \varphi ^{\ast }\theta ^{\prime }( z_{q+1},\cdots ,z_{q+r}) \vspace*{%
1mm} \\ 
\qquad \displaystyle=( \varphi ^{\ast }\omega ^{\prime }\wedge \varphi
^{\ast }\theta ^{\prime }) ( z_{1},\cdots ,z_{q+r}),%
\end{array}%
\end{equation*}%
which results (i).\newline
Let $z\in A$\emph{\ }and\emph{\ }$\omega ^{\prime }\in \Lambda ^{q}(
A^{\prime }) $. Then we obtain 
\begin{align*}
&i_{z}( \varphi ^{\ast }\omega ^{\prime }) ( z_{2},\cdots ,z_{q})=\omega
^{\prime }( \varphi ( z) ,\varphi ( z_{2}) ,\cdots ,\varphi ( z_{q}) )
=i_{\varphi ( z) }\omega ^{\prime }( \varphi ( z_{2}) ,\cdots ,\varphi (
z_{q}) ) \\
& \displaystyle=\varphi ^{\ast }( i_{\varphi ( z) }\omega ^{\prime }) (
z_{2},\cdots ,z_{q}) ,
\end{align*}%
for any $z_{2},\cdots ,z_{q}\in A$. Thus (ii) holds.\newline
Let $\omega ^{\prime }\in \Lambda ^{q}( A^{\prime }) $ and $z_{0},\cdots
,z_{q}\in A$. Then we deduce 
\begin{equation*}
\begin{array}{l}
( \varphi ^{\ast }d^{A^{\prime }}\omega ^{\prime }) ( z_{0},\cdots ,z_{q})
=( d^{A^{\prime }}\omega ^{\prime }) ( \varphi ( z_{0}) ,\cdots ,\varphi (
z_{q}) ) \vspace*{1mm} \\ 
=\overset{q}{\underset{i=0}{\sum }}( -1) ^{i}\rho ^{\prime }( \varphi ( z_{i}%
\vspace*{1mm}) ) \cdot \omega ^{\prime }( ( \varphi ( z_{0}) ,\varphi (
z_{1}) ,\cdots ,\widehat{\varphi ( z_{i}) },\cdots ,\varphi ( z_{q}) ) ) \\ 
+\underset{0\leq i<j}{\sum }( -1) ^{i+j}\cdot \omega ^{\prime }( \varphi ( [
z_{i},z_{j}] _{A}) ,\varphi ( z_{0}) ,\varphi ( z_{1}) ,\cdots ,\widehat{%
\varphi ( z_{i}) },\cdots ,\widehat{\varphi ( z_{j}) },\cdots ,\varphi (
z_{q}) ),%
\end{array}%
\end{equation*}%
and 
\begin{equation*}
\begin{array}{l}
d^{A}( \varphi ^{\ast }\omega ^{\prime }) ( z_{0},\cdots ,z_{q}) \vspace*{1mm%
} \\ 
=\overset{q}{\underset{i=0}{\sum }}( -1) ^{i}\rho ( z_{i}) \cdot ( \varphi
^{\ast }\omega ^{\prime }) ( z_{0},\cdots ,\widehat{z_{i}},\cdots ,z_{q}) 
\vspace*{1mm} \\ 
+\underset{0\leq i<j}{\sum }( -1) ^{i+j}\cdot ( \varphi ^{\ast }\omega
^{\prime }) ( [ z_{i},z_{j}] _{A},z_{0},\cdots ,\widehat{z_{i}},\cdots ,%
\widehat{z_{j}},\cdots ,z_{q}) \vspace*{1mm} \\ 
=\overset{q}{\underset{i=0}{\sum }}( -1) ^{i}\rho ( z_{i}) \cdot \omega
^{\prime }( \varphi ( z_{0}) ,\cdots ,\widehat{\varphi ( z_{i}) },\cdots
,\varphi ( z_{q}) ) \vspace*{1mm} \\ 
+\underset{0\leq i<j}{\sum }( -1) ^{i+j}\cdot \omega ^{\prime }( \varphi ( [
z_{i},z_{j}] _{A}) ,\varphi ( z_{0}) ,\varphi ( z_{1}) ,\cdots ,\widehat{%
\varphi ( z_{i}) },\cdots ,\widehat{\varphi ( z_{j}) },\cdots ,\varphi (
z_{q}) ).%
\end{array}%
\end{equation*}%
Two above equations imply (iii).
\end{proof}

Let $\omega \in \Lambda ^{q}( A)$. If $d^A\omega=0$, then we say that $%
\omega $ is \textit{closed exterior $q$-form}. Also, if there exists $%
\eta\in\Lambda^{q-1}(A)$ such that $\omega=d^{A}\eta$, then we say that $%
\omega$ is \textit{exact exterior $q$-form}. For any $q\in \overline{1,n}$
we denote by $\mathcal{Z}^{q}( A)$ and $\mathcal{B}^{q}( A)$, the set of
closed exterior $q$-forms and the set of exact exterior $q$-forms,
respectively. Indeed we have 
\begin{equation*}
\mathcal{Z}^{q}( A) =\left\{ \omega \in \Lambda ^{q}( A) |d^{A}\omega
=0\right\} ,
\end{equation*}%
\begin{equation*}
\mathcal{B}^{q}( A) =\left\{ \omega \in \Lambda ^{q}( A) |\exists \eta \in
\Lambda ^{q-1}( A)\ \ s.t.\ \ d^{A}\eta =\omega \right\}.
\end{equation*}

\begin{definition}
If $((F,\nu ,N),[,]_{F,h},(\rho ,\eta ))$ is a generalized Lie algebroid,
then the exterior differential calculus of its generalized Lie $\mathcal{F}%
(N) $-algebra is called exterior differential calculus of the generalized
Lie algebroid $((F,\nu ,N),[,]_{F,h},(\rho ,\eta ))$.
\end{definition}


\section{Interior and exterior algebraic/differential systems}

Let $( A,[ ,] _{A},\rho ) $ be a generalized Lie $\mathcal{F}$%
-algebra such that $\dim _{\mathcal{F}}A=p$ and $((F, \nu, N), [,]_{F, h},
(\rho, \eta))$ be a generalized Lie algebroid.
\begin{definition}
Any $\mathcal{F}$-submodule $E$ of the $\mathcal{F}$-module $%
A$ will be called interior algebraic system (IAS) of the
generalized Lie $\mathcal{F}$-algebra $(A,[,]_{A},\rho )$.
\end{definition}
\begin{remark}
If $E$ is an IAS of the generalized Lie $\mathcal{F}$-algebra $%
(A,[,]_{A},\rho ),$ then we obtain an $\mathcal{F}$-submodule $%
E^{0}$ of the $\mathcal{F}$-module $A^{\ast }$ such
that 
\begin{equation*}
E^{0}\overset{put}{=}\left\{ \Omega \in A^{\ast }|\Omega (u)=0,~\forall u\in
E\right\} .
\end{equation*}
\end{remark}
The $\mathcal{F}$-submodule $E^{0}$ will be called \emph{the
annihilator }$\mathcal{F}$-submodule\emph{\ of the interior algebraic system 
}$E$.
\begin{definition}
Any vector subbundle $( E,\pi ,M) $ of the pull-back vector bundle $(
h^{\ast }F,h^{\ast }\nu ,M) $ will be called an interior differential
system (IDS) of the generalized Lie algebroid 
\begin{equation*}
( ( F,\nu ,N) ,[ ,] _{F,h},( \rho ,\eta ) ) .
\end{equation*}
\end{definition}
In particular, if $h=Id_{N}=\eta $, then we obtain the definition of \emph{%
IDS} of a Lie algebroid (see $[ {5}] $).

\begin{remark}
If $( E,\pi ,M) $ is an IDS of the generalized Lie algebroid 
\begin{equation*}
( ( F,\nu ,N) ,[ ,] _{F,h},( \rho ,\eta ) ) ,
\end{equation*}%
then we obtain a vector subbundle $( E^{0},\pi ^{0},M) $ of the vector
bundle $( \overset{\ast }{h^{\ast }F},\overset{\ast }{h^{\ast }\nu },M) $
such that 
\begin{equation*}
\Gamma ( E^{0},\pi ^{0},M) \overset{put}{=}\left\{ \Omega \in \Gamma ( 
\overset{\ast }{h^{\ast }F},\overset{\ast }{h^{\ast }\nu },M) |\Omega ( S)
=0,~\forall S\in \Gamma ( E,\pi ,M) \right\}.
\end{equation*}
The vector subbundle $( E^{0},\pi ^{0},M) $ will be called the
annihilator vector subbundle of the interior differential system $( E,\pi
,M)$.
\end{remark}
Easily we can deduce the following proposition:
\begin{proposition}
If $E$ is an IAS of the generalized Lie $\mathcal{F}$-algebra $%
(A,[,]_{A},\rho )$ such that $E=\left\langle s_{1},\cdots
,s_{r}\right\rangle $, then there exist $\theta ^{r+1},\cdots ,\theta
^{p}\in A^{\ast }$ linearly independent such that $E^{0}=\left\langle \theta
^{r+1},\cdots ,\theta ^{p}\right\rangle .$
\end{proposition}

\begin{proposition}
If $( E,\pi ,M) $\ is an IDS of the generalized Lie algebroid\emph{\ }%
\begin{equation*}
( ( F,\nu ,N) ,[ ,] _{F,h},( \rho ,\eta ) ),
\end{equation*}%
such that for any $( U,\xi _{U}) \in [ \mathcal{A}_{M}] $ we have $\Gamma (
E_{\mid U},\pi ,M) =\left\langle S_{1},...,S_{r}\right\rangle $, then there
exist $\Theta ^{r+1},...,\Theta ^{p}\in \Gamma ( \overset{\ast }{h^{\ast
}F_{\mid U}},\overset{\ast }{h^{\ast }\nu },U) $\ linearly independent such
that\emph{\ }$\Gamma ( E_{\mid U}^{0},\pi ^{0},U) =\left\langle \Theta
^{r+1},...,\Theta ^{p}\right\rangle .$
\end{proposition}
The interior algebraic system $E$ of the generalized Lie $%
\mathcal{F}$-algebra\\ $(A,[,]_{A},\rho )$ is called {\it involutive}
if $[u,v]_{A}\in E,~$for any $u,v\in E$. Also, the interior differential system $( E,\pi ,M) $ of the generalized Lie
algebroid 
\begin{equation*}
( ( F,\nu ,N) ,[ ,] _{F,h},( \rho ,\eta ) ),
\end{equation*}%
is called {\it involutive} if $[ S,T] _{h^{\ast }F}\in \Gamma ( E,\pi ,M) ,~$%
for any $S,T\in \Gamma ( E,\pi ,M)$. Using these definitions, we can deduce the following properties:
\begin{proposition}
If $E$ is an IAS of the generalized Lie $\mathcal{F}$-algebra\\ $%
(A,[,]_{A},\rho )$ and $\left\{ s_{1},\cdots ,s_{r}\right\} $ is a
basis for the $\mathcal{F}$-submodule $E$ then $E$ is
involutive if and only if $[s_{a},s_{b}]_{A}\in E,~$for any $a,b\in 
\overline{1,r}.$
\end{proposition}

\begin{proposition}
Let $( E,\pi ,M) $\ be an IDS of the generalized Lie algebroid 
\begin{equation*}
( ( F,\nu ,N) ,[ ,] _{F,h},( \rho ,\eta ) ) .
\end{equation*}
If for any $( U,\xi _{U}) \in [ \mathcal{A}_{M}] $ there exists a basis $\{
S_{1},...,S_{r}\} $ for the $\mathcal{F}( M) _{\mid U}$-submodule $ \Gamma
( E_{\mid U},\pi ,U)  ,$\ then $( E,\pi ,M) $\ is involutive if
and only if\emph{\ }%
\begin{equation*}
[ S_{a},S_{b}] _{h^{\ast }F_{\mid U}}\in \Gamma ( E_{\mid U},\pi ,U) ,
\end{equation*}%
$~$for for any $a,b\in \overline{1,r}.$
\end{proposition}

\begin{theorem}
\label{Frob} (Frobenius type) Let $E$ be an interior algebraic
system of the generalized Lie $\mathcal{F}$-algebra $(A,[,]_{A},\rho )$. If $\{ \theta ^{r+1},\cdots ,\theta ^{p}\} $ is a basis for
the annihilator submodule $E^{0}$, then $E$ is
involutive if and only if there exist 
\begin{equation*}
\omega _{\beta }^{\alpha }\in \Lambda ^{1}(A),~\alpha ,\beta \in \overline{%
r+1,p},
\end{equation*}%
such that 
\begin{equation}
d^{A}\theta ^{\alpha }=\Sigma _{\beta \in \overline{r+1,p}}\omega _{\beta
}^{\alpha }\wedge \theta ^{\beta },~\alpha \in \overline{r+1,p}.
\end{equation}
\end{theorem}

\begin{proof}
Let $\{ s_{1},\cdots ,s_{r}\} $ be a basis for the $\mathcal{F}$-submodule $%
E$ and we suppose that $s_{r+1},\cdots ,s_{p}\in A$ such that 
\begin{equation*}
\{ s_{1},\cdots ,s_{r},s_{r+1},\cdots ,s_{p}\} ,
\end{equation*}%
is a basis for the $\mathcal{F}$-module $A$. Also, let $\theta
^{1},\cdots ,\theta ^{r}\in A^{\ast }$ such that 
\begin{equation*}
\left\{ \theta ^{1},\cdots ,\theta ^{r},\theta ^{r+1},\cdots ,\theta
^{p}\right\} ,
\end{equation*}%
be a basis for the $\mathcal{F}$-module $A^{\ast }$. For any $%
a,b\in \overline{1,r}$ and $\alpha ,\beta \in \overline{r+1,p}$, we have the
equalities:%
\begin{equation*}
\theta ^{a}(s_{b})=\delta _{b}^{a},\ \ \theta ^{a}(s_{\beta })=0,\ \ \theta
^{\alpha }(s_{b})=0,\ \ \theta ^{\alpha }(s_{\beta })=\delta _{\beta
}^{\alpha }.
\end{equation*}%
We remark that the set of the 2-forms 
\begin{equation*}
\left\{ \theta ^{a}\wedge \theta ^{b},\theta ^{a}\wedge \theta ^{\beta
},\theta ^{\alpha }\wedge \theta ^{\beta },~a,b\in \overline{1,r}\wedge
\alpha ,\beta \in \overline{r+1,p}\right\} ,
\end{equation*}%
is a base for the $\mathcal{F}$-module $\Lambda ^{2}(A)$.
Therefore, we have%
\begin{equation}
d^{A}\theta ^{\alpha }=\Sigma _{b<c}A_{bc}^{\alpha }\theta ^{b}\wedge \theta
^{c}+\Sigma _{b,\gamma }B_{b\gamma }^{\alpha }\theta ^{b}\wedge \theta
^{\gamma }+\Sigma _{\beta <\gamma }C_{\beta \gamma }^{\alpha }\theta ^{\beta
}\wedge \theta ^{\gamma },  \label{A1}
\end{equation}%
where, $A_{bc}^{\alpha },B_{b\gamma }^{\alpha }$ and $C_{\beta \gamma
}^{\alpha },~a,b,c\in \overline{1,r}\wedge \alpha ,\beta ,\gamma \in 
\overline{r+1,p}$ are components such that $A_{bc}^{\alpha }=-A_{cb}^{\alpha
}$ and $C_{\beta \gamma }^{\alpha }=-C_{\gamma \beta }^{\alpha }$. Using the
formula 
\begin{equation}
d^{A}\theta ^{\alpha }(s_{b},s_{c})=\rho (s_{b})(\theta ^{\alpha
}(s_{c}))-\rho (s_{c})(\theta ^{\alpha }(s_{b}))-\theta ^{\alpha
}([s_{b},s_{c}]_{A}),  \label{A2}
\end{equation}%
we obtain 
\begin{equation}
A_{bc}^{\alpha }=-\theta ^{\alpha }([s_{b},s_{c}]_{A}),~\forall (b,c\in 
\overline{1,r}\wedge \alpha \in \overline{r+1,p}).  \label{A3}
\end{equation}%
We admit that $E$ is an involutive IAS of the generalized Lie $%
\mathcal{F}$-algebra $(A,[,]_{A},\rho )$. Since $%
[s_{b},s_{c}]_{A}\in E,~\forall b,c\in \overline{1,r}$, then it results that 
$\theta ^{\alpha }([s_{b},s_{c}]_{A})=0$, where $\alpha \in \overline{r+1,p}$%
. Therefore $A_{bc}^{\alpha }=0$, and consequently 
\begin{equation*}
\begin{array}{ccl}
d^{A}\theta ^{\alpha } & = & \Sigma _{b,\gamma }B_{b\gamma }^{\alpha }\theta
^{b}\wedge \theta ^{\gamma }+\frac{1}{2}C_{\beta \gamma }^{\alpha }\theta
^{\beta }\wedge \theta ^{\gamma } \\ 
& = & (B_{b\gamma }^{\alpha }\theta ^{b}+\frac{1}{2}C_{\beta \gamma
}^{\alpha }\theta ^{\beta })\wedge \theta ^{\gamma },\ \ \ \forall \alpha
,\beta ,\gamma \in \overline{r+1,p}.%
\end{array}%
\end{equation*}%
Setting 
\begin{equation*}
\omega _{\gamma }^{\alpha }\overset{put}{=}B_{b\gamma }^{\alpha }\theta ^{b}+%
\frac{1}{2}C_{\beta \gamma }^{\alpha }\theta ^{\beta }\in \Lambda ^{1}(A),
\end{equation*}%
in the above equation the necessity condition of assertion proves. Conversely, we admit that there exist $\omega _{\beta }^{\alpha }\in \Lambda
^{1}(A)$, $\alpha ,\beta \in \overline{r+1,p}$, such that 
\begin{equation}  \label{A4}
d^{A}\theta ^{\alpha }=\Sigma _{\beta \in \overline{r+1,p}}\omega _{\beta
}^{\alpha }\wedge \theta ^{\beta }.
\end{equation}
Using the affirmations (\ref{A1}), (\ref{A2}) and (\ref{A4}) we derive that 
\begin{equation*}
A_{bc}^{\alpha }=0,~\forall b,c\in \overline{1,r}, \ \ \forall\alpha \in 
\overline{r+1,p}.
\end{equation*}
Thus using (\ref{A3}), we obtain 
\begin{equation*}
\theta ^{\alpha }( [ s_{b},s_{c}] _{A}) =0,~\forall ( b,c\in \overline{1,r}%
\wedge \alpha \in \overline{r+1,p}),
\end{equation*}
which gives us 
\begin{equation*}
[ s_{b},s_{c}] _{A}\in E,~\forall b,c\in \overline{1,r}.
\end{equation*}
Therefore from previous proposition we deduce that $E$ is
involutive.
\end{proof}

\begin{corollary}
\label{CFrob} \textbf{(}of Frobenius type)\textbf{\ } Let $( E,\pi ,M) $\ be
an IDS of the generalized Lie algebroid $( ( F,\nu ,N) ,[ ,] _{F,h},( \rho
,\eta ) ) .$\ If for any $( U,\xi _{U}) \in [ \mathcal{A}_{M}] $, there
exists the basis $\left\{ \Theta ^{r+1},...,\Theta ^{p}\right\} $ for the $%
\mathcal{F}( M) _{\mid U}$-submodule $\Gamma ( E_{\mid U}^{0},\pi ^{0},U)
 $, then $( E,\pi ,M) $ is involutive if and only if there exist 
\begin{equation*}
\Omega _{\beta }^{\alpha }\in \Lambda ^{1}( h^{\ast }F_{\mid U},h^{\ast }\nu
,U) ,~\alpha ,\beta \in \overline{r+1,p}
\end{equation*}%
such that 
\begin{equation*}
d^{h^{\ast }F}\Theta ^{\alpha }=\Sigma _{\beta \in \overline{r+1,p}}\Omega
_{\beta }^{\alpha }\wedge \Theta ^{\beta },~\alpha \in \overline{r+1,p}.
\end{equation*}
\end{corollary}

\begin{definition}
An $\mathcal{F}$-submodule $\mathcal{I} $ of the Lie $\mathcal{F}$%
-algebra $( A,[ ,] _{A}) $ such that $[ u,v] _{A}\in \mathcal{I},$
for any $u\in A$ and $v\in \mathcal{I}$, is called the ideal of the Lie $%
\mathcal{F}$-algebra $( A,[ ,] _{A})$.
\end{definition}

\begin{remark}
If $E $ is a $\mathcal{F}$-submodule of the Lie $\mathcal{F}$%
-algebra $( A,[ ,] _{A}) $ and%
\begin{equation*}
\mathcal{I}( E) \overset{put}{=}\underset{\mathcal{I~}\supseteq ~E}{\underset%
{\mathcal{I}=~ideal}{\bigcap }}\mathcal{I},
\end{equation*}%
then $ \mathcal{I}( E)  $ is an ideal of the Lie $\mathcal{F}$%
-algebra $( A,[ ,] _{A}) $ which is called the ideal generated by
the $\mathcal{F}$-submodule $ E .$
\end{remark}
\begin{definition}
An ideal $\mathcal{I} $ of the exterior algebra of the
generalized Lie $\mathcal{F}$-algebra $( A,[ ,] _{A},\rho ) $
closed under differentiation operator $d^{A},$ namely $d^{A}\mathcal{%
I\subseteq I},$ is called a \emph{differential ideal of the generalized
Lie }$\mathcal{F}$-\emph{algebra }$( A,[ ,] _{A},\rho ) .$
\end{definition}
Let $\mathcal{I}$ be a differential ideal of the generalized Lie $%
\mathcal{F}$-algebra $(A,[,]_{A},\rho )$. If there exists an \emph{%
interior algebraic system }$E$ such that for all $k\in \mathbb{N}%
^{\ast }$ and $\omega \in \mathcal{I}\cap \Lambda ^{k}(A)$ we have $\omega
(u_{1},\cdots ,u_{k})=0,$ for any $u_{1},\cdots ,u_{k}\in E,$ then we will
say that $\mathcal{I}$\emph{\ is an exterior algebraic system of
the generalized Lie }$\mathcal{F}$-algebra $(A,[,]_{A},\rho ).$

\begin{definition}
Any ideal $\mathcal{I} $ of the exterior differential algebra of
the pull-back Lie algebroid 
\begin{equation*}
( ( h^{\ast }F,h^{\ast }\nu ,M) ,[ ,] _{h^{\ast }F},( \overset{h^{\ast }F}{%
\rho },Id_{M}) )
\end{equation*}%
closed under differentiation operator $d^{h^{\ast }F},$ namely $d^{h^{\ast
}F}\mathcal{I\subseteq I},$ will be called a differential ideal of the
generalized Lie algebroid\emph{\ }$( ( F,\nu ,N) ,[ ,] _{F,h},( \rho ,\eta )
) .$
\end{definition}
Let $\mathcal{I} $ be a differential ideal of the generalized
Lie algebroid $( ( F,\nu ,N) ,[ ,] _{F,h},( \rho ,\eta ) ) $. If there
exists an IDS\emph{\ }$( E,\pi ,M) $ such that for all $k\in \mathbb{N}%
^{\ast }$ and $\omega \in \mathcal{I}\cap \Lambda ^{k}( h^{\ast }F,h^{\ast
}\nu ,M) $ we have $\omega ( u_{1},...,u_{k}) =0,$ for any $%
u_{1},...,u_{k}\in \Gamma ( E,\pi ,M) ,$ then we will say that $ \mathcal{I}%
 $\emph{\ is an exterior differential system of the generalized
Lie algebroid\ }%
\begin{equation*}
( ( F,\nu ,N) ,[ ,] _{F,h},( \rho ,\eta ) ) .
\end{equation*}
In particular, if $h=Id_{N}=\eta $, then we obtain the definition of the EDS
of a Lie algebroid (see$[ 5] $).

\begin{theorem}
\label{Cartan} (Cartan type) The interior algebraic system $E$ of
the generalized Lie $\mathcal{F}$-algebra $(A,[,]_{A},\rho )$ is
involutive, if and only if the ideal generated by the $\mathcal{F}$%
-submodule $E^{0}$ is an exterior algebraic system of the same
generalized Lie $\mathcal{F}$-algebra $(A,[,]_{A},\rho )$.
\end{theorem}

\begin{proof}
Let $E$ be an involutive interior algebraic system of the
generalized Lie $\mathcal{F}$-algebra $(A,[,]_{A},\rho )$ and let $%
\left\{ \theta ^{r+1},\cdots ,\theta ^{p}\right\} $ be a basis for the $%
\mathcal{F}$-submodule $E^{0}$. We know that 
\begin{equation*}
\mathcal{I}(E^{0})=\cup _{q\in \mathbb{N}}\left\{ \omega _{\alpha }\wedge
\theta ^{\alpha },~\left\{ \omega _{r+1},\cdots ,\omega _{p}\right\} \subset
\Lambda ^{q}(A)\right\} .
\end{equation*}
Let $q\in \mathbb{N}$ and $\left\{ \omega _{r+1},\cdots ,\omega _{p}\right\}
\subset \Lambda ^{q}(A)$. Using Theorems \ref{38} and \ref{Frob} we obtain 
\begin{equation*}
\begin{array}{ccl}
d^{A}(\omega _{\alpha }\wedge \theta ^{\alpha }) & = & d^{A}\omega _{\alpha
}\wedge \theta ^{\alpha }+(-1)\omega _{\beta }^{q+1}\wedge d^{A}\theta
^{\beta } \\ 
& = & (d^{A}\omega _{\alpha }+(-1)\omega _{\beta }^{q+1}\wedge \omega
_{\alpha }^{\beta })\wedge \theta ^{\alpha }.%
\end{array}%
\end{equation*}
Since 
\begin{equation*}
d^{A}\omega _{\alpha }+(-1)\omega _{\beta }^{q+1}\wedge \omega _{\alpha
}^{\beta }\in \Lambda ^{q+2}(A),
\end{equation*}%
then we get $d^{A}(\omega _{\beta }\wedge \theta ^{\beta })\in \mathcal{I}%
(E^{0})$, and consequently $d^{A}\mathcal{I}(E^{0})\subseteq \mathcal{I}%
(E^{0})$.

Conversely, let $E$ be an interior algebraic system of the
generalized Lie $\mathcal{F}$-algebra $(A,[,]_{A},\rho )$ such that
the $\mathcal{F}$-submodule $\mathcal{I}(E^{0})$ is an exterior
algebraic system of the generalized Lie $\mathcal{F}$-algebra $(A,[,]_{A},\rho )$. Suppose that $\left\{ \theta ^{r+1},\cdots ,\theta
^{p}\right\} $ is a basis for the $\mathcal{F}$-submodule $E^{0}$.
Since $d^{A}\mathcal{I}(E^{0})\subseteq \mathcal{I}(E^{0})$, then there
exist $\omega _{\beta }^{\alpha }\in \Lambda ^{1}(A),~\alpha ,\beta \in 
\overline{r+1,p}$, such that 
\begin{equation*}
d^{A}\theta ^{\alpha }=\Sigma _{\beta \in \overline{r+1,p}}\omega _{\beta
}^{\alpha }\wedge \theta ^{\beta }\in \mathcal{I}(E^{0}).
\end{equation*}%
Using Theorem \ref{Frob}, it results that $E$ is an involutive
interior algebraic system of the generalized Lie $\mathcal{F}$-algebra\emph{%
\ }$(A,[,]_{A},\rho )$\emph{.}
\end{proof}

\begin{corollary}
\label{CCartan}  The interior differential system $( E,\pi ,M) $ of the
generalized Lie algebroid $$( ( F,\nu ,N) ,[ ,] _{F,h},( \rho ,\eta ) ),$$ is
involutive, if and only if for any $( U,\xi _{U}) \in [ \mathcal{A}_{M}] $
 the ideal generated by the $\mathcal{F}( M) _{\mid
U} $-submodule $\Gamma ( E_{\mid U}^{0},\pi ^{0},U)  $ is an
EDS of the same generalized Lie algebroid.
\end{corollary}


\section{New directions by research}

We know that a generalized Lie $\mathcal{F}$-algebras morphism from $(A,[,]_{A},\rho )$ to $(A^{\prime },[,]_{A^{\prime }},\rho
^{\prime })$ is a Lie $\mathcal{F}$-algebras morphism $\varphi $ from $%
(A,[,]_{A}~)$ to $(A^{\prime },[,]_{A^{\prime }})$ such
that $\rho ^{\prime }\circ \varphi =\rho $.

\begin{definition}
An \emph{algebraic symplectic }$\mathcal{F}$\emph{-space} is a pair $%
((A,[,]_{A},\rho ),\omega )$ consisting by a generalized Lie $%
\mathcal{F}$-algebra $(A,[,]_{A},\rho )$ and a nondegenerate closed 
$2$-form $\omega \in \Lambda ^{2}(A).$
\end{definition}

If $((A^{\prime },[,]_{A^{\prime }},\rho ^{\prime }),\omega
^{\prime })$ is an another generalized symplectic $\mathcal{F}$-space, then
we can define the set of morphisms from $((A,[,]_{A},\rho ),\omega
) $ to $((A^{\prime },[,]_{A^{\prime }},\rho ^{\prime }),\omega
^{\prime })$ as being the set of generalized Lie $\mathcal{F}$-algebras
morphisms $\varphi $ such that $\varphi ^{\ast }(\omega ^{\prime })=\omega $. So, we can discuss about the category of algebraic symplectic $\mathcal{F}$%
-spaces as being a subcategory of the category of generalized Lie $\mathcal{F%
}$-algebras. An algebraic study of objects of this category is a new
direction by research.

It is known that the set of morphisms from $( ( F,\nu ,N) ,[ ,] _{F,h},( \rho
,\eta ) ) $ to $( ( F^{\prime },\nu ^{\prime },N^{\prime }) ,[ ,] _{F^{\prime },h^{\prime
}},( \rho ^{\prime },\eta ^{\prime }) ) $ is the set of vector bundles
morphisms $( \varphi ,\varphi _{0}) $ from $( F,\nu ,N) $ to $( F^{\prime
},\nu ^{\prime },N^{\prime }) $ such that $\varphi _{0}$ is diffeomorphism
and the modules morphism $\Gamma ( \varphi ,\varphi _{0}) $ is a Lie $%
\mathcal{F}( N) $-algebras morphism from $( \Gamma ( F,\nu ,N) ,[
,] _{F,h}) $ to $( \Gamma ( F^{\prime },\nu ^{\prime },N^{\prime }), [ ,] _{F^{\prime },h^{\prime }}) $ (see \cite{A4}). We can define \emph{%
the differential symplectic space }as being a pair 
\begin{equation*}
\begin{array}{c}
( ( ( F,\nu ,N) ,[ ,] _{F,h},( \rho ,\eta ) ) ,\omega ),
\end{array}%
\end{equation*}%
consisting of a generalized Lie algebroid $( ( F,\nu ,N) ,[ ,] _{F,h},( \rho
,\eta ) ) $ and a nondegenerate closed $2$-form $\omega \in \Lambda ^{2}(
F,\nu ,N) .$ If $( ( ( F^{\prime },\nu ^{\prime },N^{\prime }) ,[ ,] _{F^{\prime },h^{\prime
}},( \rho ^{\prime },\eta ^{\prime }) ) ,\omega ^{\prime })$, 
is an another differential symplectic space, then we can define the set of
morphisms from $( ( ( F,\nu ,N) ,[ ,] _{F,h},( \rho ,\eta ) ) ,\omega ) $ to 
$( ( ( F^{\prime },\nu ^{\prime },N^{\prime }) ,[ ,] _{F^{\prime },h^{\prime
}},( \rho ^{\prime },\eta ^{\prime }) ) ,\omega ^{\prime }) $ as being the
set of generalized Lie algebroids morphisms $( \varphi ,\varphi _{0}) $ such
that 
\begin{equation*}
( \varphi ,\varphi _{0}) ^{\ast }( \omega ^{\prime }) =\omega .
\end{equation*}
So, we can discuss about the category of differential symplectic spaces as
being a subcategory of the category of generalized Lie algebroids. The study
of the geometry of objects of this category is an another direction by
research.

\section{Acknowledgment}

\addcontentsline{toc}{section}{Acknowledgment}

The first author would like to thank R\u{a}dine\c{s}ti-Gorj Cultural
Scientifique Society for financial support. In memory of Prof. Dr. Gheorghe
RADU and Acad. Dr. Doc. Cornelius RADU.

\end{document}